\documentclass[12pt]{amsart}
\usepackage{amsmath}
\usepackage{amsthm}
\usepackage{amsfonts}
\usepackage{amssymb}
\usepackage{mathrsfs}
\usepackage{graphicx}
\usepackage{stmaryrd}
\usepackage{dsfont}
\usepackage[all]{xy}
\usepackage[shortlabels]{enumitem}
\usepackage[dvipsnames]{xcolor}
\usepackage[margin=1in]{geometry} 
\usepackage[bookmarks, bookmarksdepth=2, colorlinks=true, linkcolor=blue, citecolor=blue, urlcolor=blue]{hyperref}
\usepackage{eucal}
\usepackage{contour}
\usepackage[normalem]{ulem}
\usepackage{tikz}
\usetikzlibrary{decorations.pathreplacing}

\makeatletter
\@addtoreset{equation}{section}
\makeatother

\numberwithin{equation}{section}
\newtheorem{theorem}[equation]{Theorem}

\newtheorem{proposition}[equation]{Proposition}
\newtheorem{lemma}[equation]{Lemma}
\newtheorem{corollary}[equation]{Corollary}

\theoremstyle{definition}
\newtheorem{rmk}[equation]{Remark}
\newenvironment{remark}[1][]{\begin{rmk}[#1] \pushQED{\qed}}{\popQED \end{rmk}}
\newtheorem{eg}[equation]{Example}
\newenvironment{example}[1][]{\begin{eg}[#1] \pushQED{\qed}}{\popQED \end{eg}}
\newtheorem{defnaux}[equation]{Definition}
\newenvironment{definition}[1][]{\begin{defnaux}[#1]\pushQED{\qed}}{\popQED \end{defnaux}}

\newcommand{\arxiv}[1]{\href{http://arxiv.org/abs/#1}{{\tiny\tt arXiv:#1}}}
\newcommand{\DOI}[1]{\href{http://doi.org/#1}{\color{purple}{\tiny\tt DOI:#1}}}

\newcommand{\cC}{\mathcal{C}}

\newcommand{\cD}{\mathcal{D}}

\newcommand{\bG}{\mathbf{G}}

\newcommand{\bI}{\mathbf{I}}

\newcommand{\bJ}{\mathbf{J}}

\newcommand{\bK}{\mathbf{K}}

\newcommand{\cN}{\mathcal{N}}

\newcommand{\cP}{\mathcal{P}}

\newcommand{\cQ}{\mathcal{Q}}

\newcommand{\bR}{\mathbf{R}}

\newcommand{\bS}{\mathbf{S}}

\newcommand{\fS}{\mathfrak{S}}

\newcommand{\cT}{\mathcal{T}}

\newcommand{\bZ}{\mathbf{Z}}

\newcommand{\rf}{\mathrm{f}}

\let\ol\overline
\let\ul\underline

\newcommand{\defn}[1]{\emph{#1}}
\renewcommand{\phi}{\varphi}
\renewcommand{\emptyset}{\varnothing}
\newcommand{\lw}{{\textstyle \bigwedge}}

\DeclareMathOperator{\coker}{coker}

\DeclareMathOperator{\End}{End}

\DeclareMathOperator{\Aut}{Aut}

\DeclareMathOperator{\Mod}{Mod}
\DeclareMathOperator{\Ind}{Ind}

\DeclareMathOperator{\Hom}{Hom}

\DeclareMathOperator{\Ext}{Ext}

\DeclareMathOperator{\Rep}{Rep}
\DeclareMathOperator{\vol}{vol}

\DeclareMathOperator{\utr}{\ul{tr}}

\DeclareMathOperator{\Res}{Res}
\newcommand{\GL}{\mathbf{GL}}

\newcommand{\ev}{\mathrm{ev}}

\newcommand{\bbone}{\mathds{1}}
\newcommand{\uotimes}{\mathbin{\ul{\otimes}}}

\newcommand{\wa}{{\bullet}}
\newcommand{\wb}{{\circ}}

\contourlength{1pt}

\contourlength{0.8pt}
\newcommand{\myuline}[1]{%
  \uline{\phantom{#1}}%
  \llap{\contour{white}{#1}}%
}
\DeclareMathOperator{\uRep}{\text{\myuline{\rm Rep}}}
\DeclareMathOperator{\uPerm}{\ul{Perm}}

\title{The circular Delannoy category}
\author{Nate Harman}
\author{Andrew Snowden}
\author{Noah Snyder}

\date{March 18, 2023}

\begin{document}

\begin{abstract}
Let $G$ (resp.\ $H$) be the group of orientation preserving self-homeomorphisms of the unit circle (resp.\ real line). In previous work, the first two authors constructed pre-Tannakian categories $\uRep(G)$ and $\uRep(H)$ associated to these groups. In the predecessor to this paper, we analyzed the category $\uRep(H)$ (which we named the ``Delannoy category'') in great detail, and found it to have many special properties. In this paper, we study $\uRep(G)$. The primary difference between these two categories is that $\uRep(H)$ is semi-simple, while $\uRep(G)$ is not; this introduces new complications in the present case. We find that $\uRep(G)$ is closely related to the combinatorics of objects we call Delannoy loops, which seem to have not previously been studied.
\end{abstract}

\maketitle
\tableofcontents

\section{Introduction}

\subsection{Background}

A common practice in mathematics is to identify the key features of a set of examples, and take these as the defining axioms for a more general class of objects. In the setting of tensor categories, representation categories of algebraic groups are a fundamental set of examples. Taking their key features\footnote{Of course, pre-Tannakian categories are not the only generalization of classical representation categories. For instance, one can relax the symmetry of the tensor product to a braiding. However, of the different generalizations, pre-Tannakian categories seem to be closest to group representations.} as axioms leads to the notion of \defn{pre-Tannakian category} (see \cite[\S 2.1]{ComesOstrik} for a precise definition). Understanding these categories is an important problem within the field of tensor categories. The last twenty years has seen important progress, but the picture is still far from clear.

One difficulty in the study of pre-Tannakian categories is that few examples are known, beyond those coming from (super)groups. In recent work \cite{repst}, the first two authors gave a general construction of pre-Tannakian categories. Given an oligomorphic group $G$ equipped with a piece of data called a measure $\mu$, we constructed a tensor category $\uPerm(G; \mu)$ analogous to the category of permutation representations. This category sometimes has a pre-Tannakian abelian envelope $\uRep(G; \mu)$. The basic ideas are reviewed in \S \ref{s:olig}. In forthcoming work \cite{discrete}, we show that any discrete pre-Tannakian category (i.e., one generated by an \'etale algebra) comes from this construction. If pre-Tannakian categories are an abstraction of algebraic groups, then the discrete ones are an abstraction of finite groups. Thus this is a very natural class of examples.

The simplest example of an oligomorphic group is the infinite symmetric group. In this case, the theory of \cite{repst} leads to Deligne's interpolation category $\uRep(\fS_t)$ \cite{Deligne3}; we note that Deligne's work predated (and motivated) \cite{repst}. This category (and related examples) has received much attention in the literature, e.g., \cite{ComesOstrik1, ComesOstrik, Etingof1, Etingof2, EntovaAizenbudHeidersdorf, Harman, Harman2, Knop, Knop2}.

Currently, there are only (essentially) three known examples of pre-Tannakian categories coming from oligomorphic groups that fall outside the setting of Deligne interpolation:
\begin{enumerate}
\item The group $H$ of all orientation-preserving self-homeomorphisms of the real line carries four measures; one of these is known to yield a pre-Tannakian category $\uRep(H)$.
\item The group $G$ of all orientation-preserving self-homeomorphisms of the unit circle carries a unique measure, and this yields a pre-Tannakian category $\uRep(G)$.
\item The automorphism group of the universal boron tree carries two measures, one of which is known to yield a pre-Tannakian category; see \cite[\S 18]{repst}.
\end{enumerate}
We do not know if the other three measures for $H$, or the other measure in case (c), can be associated to pre-Tannakian categories. This is an important problem.

Deligne's categories have been well-studied, but since they are closely tied to finite groups they may not be representative of general pre-Tannakian categories. We therefore feel it is important to carefully study the above three examples. In our previous paper \cite{line}, we studied $\uRep(H)$ in detail. We named it the \emph{Delannoy category}, due to its connection to Delannoy paths, and found that it has numerous special properties. In this paper, we study $\uRep(G)$. While $\uRep(G)$ is closely related to $\uRep(H)$, the latter is semi-simple and the former is not, and this introduces a number of new difficulties. In the rest of the introduction, we recall some background on these categories and then state our main results.

\subsection{Representations of $H$}

We briefly recall the construction of $\uRep(H)$, and some results from \cite{line}. See \S \ref{s:olig} and \S \ref{s:setup} for details. We fix an algebraically closed field $k$ of characteristic~0 in what follows.

Let $\bR^{(n)}$ be the subset of $\bR^n$ consisting of tuples $(x_1, \ldots, x_n)$ with $x_1<\cdots<x_n$. The group $H$ acts transitively on this set. We say that a function $\phi \colon \bR^{(n)} \to k$ is \defn{Schwartz} if it assumes finitely many values and its level sets can be defined by first-order formulas using $<$, $=$, and finitely many real constants; equivalently, the stabilizer in $H$ of $\phi$ is open in the natural topology. We define the \defn{Schwartz space} $\cC(\bR^{(n)})$ to be the space of all Schwartz functions. An important feature of this space is that there is a notion of integral, namely, integration with respect to Euler characteristic (as defined by Schapira and Viro \cite{Viro}).

We define a category $\uPerm(H)$ by taking the objects to be formal sums of $\cC(\bR^{(n)})$'s, and the morphisms between basic objects to be $H$-invariant integral operators. This category admits a tensor structure $\uotimes$ by
\begin{displaymath}
\cC(\bR^{(n)}) \uotimes \cC(\bR^{(m)}) = \cC(\bR^{(n)} \times \bR^{(m)}),
\end{displaymath}
where on the left side we decompose $\bR^{(n)} \times \bR^{(m)}$ into orbits (which have the form $\bR^{(s)}$ for various $s$), and then take the corresponding formal sum of Schwartz spaces. The category $\uRep(H)$ is equivalent to (the ind-completion of) the Karoubian envelope of $\uPerm(H)$. This characterization of $\uRep(H)$ relies on non-trivial results of \cite{repst}.

In \cite{line}, we determined the structure of $\uRep(H)$ in greater detail. The construction and classification of its simple objects will play an important role in this paper. A \defn{weight} is a word in the alphabet $\{\wa, \wb\}$. Given a weight $\lambda$, we define $L_{\lambda}$ to be the submodule of $\cC(\bR^{(n)})$ generated by a certain explicit family of Schwartz functions (see \S \ref{ss:RepH}). We showed that these account for all the simple objects of $\uRep(H)$. As we mentioned above, the category $\uRep(H)$ is semi-simple; this follows from general results from \cite{repst}.

\subsection{Representations of $G$}

We now recall the construction of $\uRep(G)$, which is similar to the above. Let $\bS^{\{n\}}$ be the subset of $\bS^n$ consisting of tuples $(x_1, \ldots, x_n)$ such that $x_i$ is strictly between $x_{i-1}$ and $x_{i+1}$ (in counterclockwise cyclic order) for all $i \in \bZ/n$. The group $G$ acts transitively on $\bS^{\{n\}}$, and this space plays an analogous role to $\bR^{(n)}$ from the $H$ theory. We have a notion of Schwartz space $\cC(\bS^{\{n\}})$ and integration just as before, and this leads to a category\footnote{The category $\uPerm(G)$ defined in \cite{repst} is a little bigger than the $\uPerm^{\circ}(G)$ category defined here.} $\uPerm^{\circ}(G)$. The category $\uRep(G)$ is equivalent to the (ind-completion of) the abelian envelope of $\uPerm^{\circ}(G)$. Again, this relies on non-trivial results from \cite{repst}. Since $\uRep(G)$ is not semi-simple, we can no longer simply take the Karoubian envelope to obtain the abelian envelope.

Fix a point $\infty \in \bS$ and identify $\bR$ with $\bS \setminus \{\infty\}$. Then $H$ is identified with the stabilizer of $\infty$ in $G$, and as such is an open subgroup of $G$. There is thus a restriction functor $\uRep(G) \to \uRep(H)$, which admits a two-sided adjoint called induction. Since we already know the structure of $\uRep(H)$ from \cite{line}, restriction and induction will be very useful in the study of $\uRep(G)$.

\subsection{Results} \label{ss:results}

We now discuss our main results about the category $\uRep(G)$.

\textit{(a) Classification of simples.} Given a non-empty weight $\lambda$, we define $\Delta_{\lambda}$ to be the $\ul{G}$-submodule of $\cC(\bS^{\{n\}})$ generated by an explicit set of Schwartz functions. The group $\bZ/n$ acts on $\cC(\bS^{\{n\}})$ by cyclicly permuting co-ordinates, and $\Delta_{\lambda}$ is stable under the action of the subgroup $\Aut(\lambda)$ consisting of cyclic symmetries of $\lambda$. We let $\Delta_{\lambda,\zeta}$ be the $\zeta$-eigenspace of the generator of $\Aut(\lambda)$, where $\zeta$ is an appropriate root of unity. We show that $\Delta_{\lambda,\zeta}$ has a unique maximal proper $\ul{G}$-submodule, and thus a unique simple quotient $M_{\lambda,\zeta}$. We show that the $M_{\lambda,\zeta}$, together with the trivial module, account for all simple $\ul{G}$-modules. The simples $M_{\lambda,\zeta}$ and $M_{\mu,\omega}$ are isomorphic if and only if $\mu$ is a cyclic shift of $\lambda$ and $\omega=\zeta$.

We say that $(\lambda,\zeta)$ is \defn{special} if $\lambda$ consists only of $\wa$'s, or only of $\wb$'s, and $\zeta=(-1)^{\ell(\lambda)+1}$; otherwise, we say that $(\lambda,\zeta)$ is \defn{generic}. In the generic case, $\Delta_{\lambda,\zeta}$ is in fact already simple. In the special case, the module $\Delta_{\lambda,\zeta}$ has length two. The generic simples have categorical dimension~0, while the special simples have dimension $\pm 1$.

\textit{(b) Structure of general modules.} Let $\cC=\uRep^{\rf}(G)$. Define $\cC_{\rm gen}$ (resp.\ $\cC_{\rm sp}$) to be the full subcategory of $\cC$ spanned by modules whose simple constituents are generic (resp.\ special). We prove the following three statements:
\begin{itemize}
\item The category $\cC$ is the direct sum of the subcategories $\cC_{\rm gen}$ and $\cC_{\rm sp}$.
\item The category $\cC_{\rm gen}$ is semi-simple.
\item The category $\cC_{\rm sp}$ is equivalent to the category of finitely generated $\bZ$-graded $R$-modules, where $R=k[x,y]/(x^2,y^2)$, and $x$ and $y$ have degrees $-1$ and $+1$.
\end{itemize}
We note that the above results are just statements about $\cC$ as an abelian category, and ignore the tensor structure. Using the above results, we classify the indecomposable $\ul{G}$-modules; they are ``zigzag modules,'' which appear frequently across representation theory.

\textit{(c) Branching rules.}
We determine the induction and restriction rules between $G$ and $H$. For a simple $L_\lambda$ in $H$, we determine the decomposition of its induction $I_{\lambda}$ into indecomposable projectives (Proposition~\ref{prop:I-decomp}); this decomposition is multiplicity-free, and at most one of the projectives belongs to the special block (the rest are generic simples). Since $\uRep(H)$ is semisimple, this determines the induction of any object in $\uRep(H)$. Additionally, for simple object $M_{\lambda,\zeta}$ in $\uRep(G)$ we determine the irreducible decomposition of $\Res^G_H(M_{\lambda,\zeta})$.  This determines the restriction of any object in $\uRep(G)$, provided we know its simple constituents.

\textit{(d) Semisimplification.} The \defn{semisimplification} of a symmetric tensor category is the category obtained by killing the so-called negligible morphisms; see \cite{EtingofOstrik} for general background. We show that the semisimplification of $\uRep^{\rf}(G)$ is equivalent (as a symmetric tensor category) to the category of bi-graded vector spaces equipped with a symmetric structure similar to that of super vector spaces. We note that the final answer and the proof are quite similar to the computation of the semisimplification of $\Rep(\GL(1|1))$ (see \cite{Heidersdorf}); it would be interesting if there were a conceptual explanation for this.

\textit{(e) Loop model.} We show that $\uRep(G)$ is closely related to a category constructed from combinatorial objects we introduce called \defn{Delannoy loops}. These are similar to the well-known Delannoy paths, but (as far as we can tell) have not been previously studied.

\subsection{Notation}

We list some of the most important notation here:
\begin{description}[align=right,labelwidth=2.25cm,leftmargin=!]
\item[ $k$ ] the coefficient field (algebraically closed of characteristic~0)
\item [$\bbone$ ] the trivial representation
\item[ $\bS$ ] the circle
\item[ $x<y<z$  ] the cyclic order on $\bS$, a ternary relation
\item[ $\bR$ ] the real line, identified with $\bS \setminus \{\infty\}$ (see \S \ref{ss:groups})
\item[ $\bS^{\{n\}}$ ] the subset of $\bS^n$ consisting of cyclicly ordered tuples (see \S \ref{ss:groups})
\item[ $\bR^{(n)}$ ] the subset of $\bR^n$ consisting of totally ordered tuples (see \S \ref{ss:groups})
\item[ $G$ ] the group $\Aut(\bS,<)$ (except in \S \ref{s:olig})
\item[ $H$ ] the group $\Aut(\bR,<)$ (except in \S \ref{s:olig})
\item[ $G(b)$ ] the subgroup of $G$ fixing each element of $b$
\item[ {$G[b]$} ] the subgroup of $G$ fixing $b$ as a set
\item[ $\sigma$ ] the generator of $\bZ/n$, which acts on $\bS^{\{n\}}$ by permuting the subscripts (see \S \ref{ss:groups}) and on weights of length $n$ by cyclically permuting the letters (see \S \ref{ss:weights})
\item[ $\tau$ ] the standard generator of $G[b]/G(b)$ (see \S \ref{ss:groups})
\item[ $\lambda$ ] a weight (see \S \ref{ss:weights})
\item[ $\gamma_i$ ] the $i$th cyclic contraction (see \S \ref{ss:weights})
\item[ $g(\lambda)$ ] the order of $\Aut(\lambda)$ (see \S \ref{ss:weights})
\item[ $N(\lambda)$ ] $\ell(\lambda)/g(\lambda)$ (see \S \ref{ss:weights})
\item[ $L_{\lambda}$ ] a simple $H$-module (see \S \ref{ss:RepH})
\item[ $(-)^{\dag}$ ] the transpose functor (see \S \ref{ss:transp})
\item[ $I_{\lambda}$ ] the induction of $L_{\lambda}$ from $H$ to $G$ (see \S \ref{ss:ind})
\item[ $\Delta_{\lambda}$ ] a standard $G$-module (see Definition~\ref{defn:std})
\item[ $M_{\lambda,\zeta}$ ] a simple $G$-module (see Definition~\ref{defn:M})
\item[ $\pi(n)$ ] the weight with all $\wa$'s or $\wb$'s (see \S \ref{ss:fine-intro})
\end{description}

\subsection*{Acknowledgments}

NS was supported in part by NSF grant DMS-2000093.

\section{Generalities on oligomorphic groups} \label{s:olig}

In this section, we recall some key definitions and results from \cite{repst}. Since we already provided an overview of this theory in \cite[\S 2]{line}, we keep this discussion very brief. We close with a short discussion of Mackey theory, which is new.

\subsection{Admissible groups}

An \defn{admissible group} is a topological group $G$ that is Hausdorff, non-archimedean (open subgroups form a neighborhood basis of the identity), and Roelcke pre-compact (if $U$ and $V$ are open subgroups then $U \backslash G /V$ is finite). An \defn{oligomorphic group} is a permutation group $(G, \Omega)$ such that $G$ has finitely many orbits on $\Omega^n$ for all $n \ge 0$. Suppose that $(G,\Omega)$ is oligomorphic. For a finite subset $a$ of $\Omega$, we let $G(a)$ be the subgroup of $G$ fixing each element of $a$. The $G(a)$'s form a neighborhood basis for an admissible topology on $G$. This is the main source of admissible groups.

Fix an admissible group $G$. We say that an action of $G$ on a set is \defn{smooth} if every stabilizer is open, and \defn{finitary} if there are finitely many orbits. We use the term ``$G$-set'' for ``set equipped with a smooth action of $G$.'' A $\hat{G}$-set is a set equipped with a smooth action of some (unspecified) open subgroup of $G$; shrinking the subgroup does not change the $\hat{G}$-set.

\subsection{Integration} \label{ss:olig-int}

Fix an admissible group $G$ and a commutative ring $k$. A \defn{measure} for $G$ with values in $k$ is a rule assigning to each finitary $\hat{G}$-set $X$ a value $\mu(X)$ in $k$ such that a number of axioms hold; see \cite[\S 2.2]{line}. 

Let $X$ be a $G$-set. A function $\phi \colon X \to k$ is called \defn{Schwartz} if it is smooth (i.e., its stabilizer in $G$ is open) and has finitary support. We write $\cC(X)$ for the space of all Schwartz functions on $X$, which we call \defn{Scwhartz space}. Suppose that $\phi$ is a Schwartz function on $X$, and we have a measure $\mu$ for $G$. Let $a_1, \ldots, a_n$ be the non-zero values attained by $\phi$, and let $X_i=\phi^{-1}(a_i)$. We define the \defn{integral} of $\phi$ by
\begin{displaymath}
\int_X \phi(x) dx = \sum_{i=1}^n a_i \mu(X_i).
\end{displaymath}
Integration defines a $k$-linear map $\cC(X) \to k$. More generally, if $Y$ is a second $G$-set and $f \colon X \to Y$ is a smooth function (i.e., equivariant for some open subgroup) then there is a push-forward map $f_* \colon \cC(X) \to \cC(Y)$ defined by integrating over fibers.

There are several niceness conditions that can imposed on a measure $\mu$. We mention a few here. We say that $\mu$ is \defn{regular} if $\mu(X)$ is a unit of $k$ for all transitive $G$-sets $X$. We say that $\mu$ is \defn{quasi-regular} if there is some open subgroup $U$ such that $\mu$ restricts to a regular measure on $U$. Finally, there is a more technical condition called \defn{property~(P)} that roughly means $\mu$ is valued in some subring of $k$ that has enough maps to fields of positive characteristic; see \cite[Definition~7.17]{repst}.

\subsection{Representations} \label{ss:olig-rep}

Fix an admissible group $G$ with a $k$-valued measure $\mu$. We assume for simplicity that $k$ is a field and $\mu$ is quasi-regular and satisfies property~(P), though it is possible to be more general. In \cite[\S 10]{repst}, we introduce the \defn{completed group algebra} $A(G)$ of $G$. As a $k$-vector space, $A(G)$ is the inverse limit of the Schwartz spaces $\cC(G/U)$ over open subgroups $U$. The multiplication on $A(G)$ is defined by convolution; this uses integration, and hence depends on the choice of measure $\mu$. See \cite[\S 10.3]{repst} or \cite[\S 2.4]{line} for details. An $A(G)$-module $M$ is called \defn{smooth} if the action of $A(G)$ is continuous, with respect to the discrete topology on $M$; concretely, this means that for $x \in M$ there is some open subgroup $U$ of $G$ such that the action of $A(G)$ on $x$ factors through $\cC(G/U)$. For any $G$-set $X$, the Schwartz space $\cC(X)$ carries the structure of a smooth $A(G)$-module in a natural manner.

We define $\uRep(G)$ to be the category of smooth $A(G)$-modules; this is a Grothendieck abelian category. In \cite[\S 12]{repst}, we construct a tensor product $\uotimes$ on $\uRep(G)$. This tensor product is $k$-bilinear, bi-cocontinuous, and satisfies $\cC(X) \uotimes \cC(Y) = \cC(X \times Y)$, and these properties uniquely characterize it; additionally, it is exact. In \cite[\S 13]{repst}, we show that every object of $\uRep(G)$ is the union of its finite length subobjects, and that the category $\uRep^{\rf}(G)$ of finite length objects is pre-Tannakian. Moreover, we show that $\uRep(G)$ is semi-simple if $\mu$ is regular.

\subsection{Induction and restriction}

Let $G$ and $\mu$ be as above, and let $H$ be an open subgroup of $G$. Then $\mu$ restricts to a quasi-regular measure on $H$ which still satisfies property~(P). There is a natural algebra homomorphism $A(H) \to A(G)$ that induces a restriction functor
\begin{displaymath}
\Res^G_H \colon \uRep(G) \to \uRep(H).
\end{displaymath}
Suppose now that $N$ is a smooth $A(H)$-module. Define $\Ind_H^G(N)$ to be the space of all functions $\phi \colon G \to N$ satsifying the following two conditions:
\begin{itemize}
\item $\phi$ is left $H$-equivariant, i.e., $\phi(hg)=h\phi(g)$ for all $h \in H$ and $g \in G$.
\item $\phi$ is right $G$-smooth, i.e., there is an open subgroup $U$ of $G$ such that $\phi(gu)=\phi(g)$ for all $g \in G$ and $u \in U$.
\end{itemize}
In \cite[\S 2.5]{line}, we defined a natural smooth $A(G)$-module structure on this space. We call $\Ind_H^G(N)$ the \defn{induction} of $N$. Induction defines a functor
\begin{displaymath}
\Ind_H^G \colon \uRep(H) \to \uRep(G).
\end{displaymath}
In \cite[Proposition~2.13]{line}, we showed that induction is left and right adjoint to restriction (and thus continuous, co-continuous, and exact) and preserves finite length objects.

We now formulate an analog of Mackey's theorem:

\begin{proposition} \label{prop:mackey}
Let $H$ and $K$ be open subgroups of $G$. For a smooth $H$-module $N$, we have a natural isomorphism of $A(K)$-modules
\begin{displaymath}
\Res^G_K(\Ind_H^G(N)) = \bigoplus_{g \in H \backslash G/K} \Ind^K_{H^g \cap K}(\Res^{H^g}_{H^g \cap K}(N^g)).
\end{displaymath}
Here $H^g=gHg^{-1}$ and $N^g$ is the $A(H^g)$-module with underlying vector space $N$ and action obtained by twisting the action on $N$ with conjugation by $g$.
\end{proposition}

\begin{proof}
Write $G=\bigsqcup_{i=1}^n Hx_iK$; note that there are finitely many double cosets since $G$ is admissible. Suppose that $\phi$ is an element of $\Ind_H^G(N)$. Define $\phi_i \colon K \to N^{x_i}$ by $\phi_i(g)=\phi(x_i g)$. One easily sees that $\phi \mapsto (\phi_i)_{1 \le i \le n}$ defines an isomorphism of $k$-vector spaces
\begin{displaymath}
\Res^G_K(\Ind_H^G(N)) \to \bigoplus_{i=1}^n \Ind^K_{H^{x_i} \cap K}(\Res^{H^{x_i}}_{H^{x_i} \cap K}(N^{x_i})).
\end{displaymath}
One then verifies that this map is $A(K)$-linear; we omit the details.
\end{proof}

\section{Set-up and basic results} \label{s:setup}

In this section, we introduce the groups $G=\Aut(\bS,<)$ and $H=\Aut(\bR,<)$ and review their basic structure. We recall some results on the representation theory of $H$ from \cite{line}, and prove a few simple results about representations of $G$.

\subsection{Automorphisms of the circle and line} \label{ss:groups}

Let $\bS$ be the circle; to be definitive, we take $\bS$ to be the unit circle in the plane. We let $<$ be the cyclic order on $\bS$ defined by the ternary relation $x<y<z$ if $y$ is strictly between $x$ and $z$ when one moves from $x$ to $z$ counterclockwise. (Note that this notation can be confusing since this ternary relation is not made up of two binary relations!) We let $\infty$ be the north pole of $\bS$. The set $\bS \setminus \{\infty\}$ is totally ordered by $x<y$ if $\infty<x<y$.  As a totally ordered set, it is isomorphic to the real line $\bR$, and we identify $\bS \setminus \{\infty\}$ with $\bR$ in what follows. 

We let $G=\Aut(\bS,<)$ be the group of permutations of $\bS$ preserving the cyclic order; equivalently, $G$ is the group of orientation-preserving homeomorphisms of $\bS$. We let $H=\Aut(\bR,<)$ be the group of permutations of $\bR$ preserving the total order; equivalently, $H$ is the group of orientation-preserving homeomorphisms of $\bR$. The group $H$ is identified with the stabilizer of $\infty$ in $G$. Both $G$ and $H$ are oligomorphic permutation groups.

Let $a$ be an $n$-element subset of $\bR$. Let $H(a)$ be the subgroup of $H$ fixing each element of $a$; this coincides with the subgroup $H[a]$ fixing $a$ as a set. The $H(a)$'s define the admissible topology on $H$. In fact, every open subgroup of $H$ is of the form $H(a)$ for some $a$ \cite[Proposition~17.1]{repst}. For a totally ordered set $I$ isomorphic to $\bR$, let $H_I=\Aut(I,<)$, which is isomorphic to $H$. Let $I_1, \ldots, I_n$ be the connected components of $\bR \setminus a$, listed in order. Then $H(a)$ preserves these intervals, and the natural map $H(a) \to H_{I_1} \times \cdots \times H_{I_n}$ is an isomorphism; thus $H(a)$ is isomorphic to $H^n$.

Let $a$ be an $n$-element subset of $\bS$, and enumerate $a$ as $\{a_1,\ldots,a_n\}$ in cyclic order. Let $G(a)$ (resp.\ $G[a]$) be the subgroup of $G$ fixing each element of $a$ (resp.\ the set $a$). Let $I_1, \ldots, I_n$ be the connected components of $\bS \setminus a$, where $I_i$ is between $a_i$ and $a_{i+1}$. Then $G(a)$ preserves each $I_i$ and the natural map $G(a) \to H_{I_1} \times \cdots \times H_{I_n}$ is an isomorphism; thus $G(a) \cong H^n$. The group $G[a]$ preserves the set $\{I_1, \ldots, I_n\}$, and in fact cyclically permutes it; one thus finds $G[a] \cong \bZ/n \ltimes H^n$, where $\bZ/n$ cyclically permutes the factors. In particular, $G[a]/G(a) \cong \bZ/n$. This group has a standard generator $\tau$, which satisfies $\tau(a_i)=a_{i-1}$.

One can show that if $U$ is any open subgroup of $G$ then there is some $a \in \bS^{\{n\}}$ such that $G(a) \subset U \subset G[a]$; we omit the proof as we will not use this.

We let $\bR^{(n)}$ be the subset of $\bR^n$ consisting of tuples $(x_1, \ldots, x_n)$ with $x_1<\cdots<x_n$. This is a transitive $H$-set, and these account for all transitive $H$-sets. We let $\bS^{\{n\}}$ be the subset of $\bS^n$ consisting of tuples $(x_1, \ldots, x_n)$ with $x_1<\cdots<x_n<x_1$. This is a transitive $G$-set.  We define $\sigma \colon \bS^{\{n\}} \to \bS^{\{n\}}$ by $\sigma(x_1, \ldots, x_n)=(x_n, x_1, \ldots, x_{n-1})$. This defines an action of $\bZ/n$ on $\bS^{\{n\}}$ that commutes with the action of $G$. It follows from the classification of open subgroups of $G$ that every transitive $G$-set has the form $\bS^{\{n\}}/\Gamma$ for some $n$ and some subgroup $\Gamma$ of $\bZ/n$.

\subsection{Integration}

Fix an algebraically closed field $k$ of characteristic~0 for the remainder of the paper. By \cite[Theorem~17.7]{repst}, there is a unique $k$-valued measure $\mu$ for $H$ satisfying
\begin{displaymath}
\mu(I_1^{(n_1)} \times \cdots \times I_r^{(n_r)}) = (-1)^{n_1+\cdots+n_r},
\end{displaymath}
where $I_1, \ldots, I_r$ are disjoint open intervals in $\bR$, and $I^{(n)}$ is defined just like $\bR^{(n)}$. Note that $I_1^{(n_1)} \times \cdots \times I_r^{(n_r)}$ is naturally an $\hat{H}$-set; in fact, every $\hat{H}$-set is isomorphic to a disjoint union of ones of this form. We call this measure the \defn{principal measure} for $H$. (There are three other $k$-valued measures for $H$, but they will not be relevant to us.) We write $\vol(X)$ in place of $\mu(X)$ in what follows.

The $\hat{H}$-set $I_1^{(n_1)} \times \cdots \times I_r^{(n_r)}$ is naturally a smooth manifold, and its volume (with respect to the principal measure) is its compactly supported Euler characteristic. It follows that the integration with respect to this measure (as defined in \S \ref{ss:olig-int}) coincides with integration with respect to Euler characteristic, as developed by Schapira and Viro (see, e.g., \cite{Viro}).

The principal measure for $H$ uniquely extends to a $k$-valued measure for $G$ (see \cite[Theorem~17.13]{repst}, or \cite{colored} for more details); in fact, its extension is the unique $k$-valued measure for $G$. This measure can again be described using Euler characteristic; in particular, we find $\vol(\bS)=0$.

\subsection{Representation categories}

We let $\uRep(H)$ and $\uRep(G)$ be the representation categories defined in \S \ref{ss:olig-int} for $H$ and $G$ with respect to the principal measure. This measure is clearly regular on $H$, and satisfies property~(P) (as it is $\bZ$-valued), and so $\uRep^{\rf}(H)$ is a semi-simple pre-Tannakian category. The principal measure is not regular on $G$, since $\vol(\bS)=0$. However, it is quasi-regular (because it is regular on $H$) and still satisfies (P), and so $\uRep^{\rf}(G)$ is a pre-Tannakian category. It is easily seen that $\uRep(G)$ is not semi-simple: indeed, the natural map $\epsilon \colon \cC(\bS) \to \bbone$ does not split, as there is a unique map $\eta \colon \bbone \to \cC(\bS)$ (up to scalars) and $\epsilon \circ \eta=\vol(\bS)=0$. (Note that $\epsilon(\phi)=\int_{\bS} \phi(x) dx$ and $\eta$ is the inclusion of constant functions.)

We use the term ``$\ul{G}$-module'' for ``smooth $A(G)$-module'' and the adjective ``$\ul{G}$-linear'' for ``$A(G)$-linear.'' Thus the objects of $\uRep(G)$ are $\ul{G}$-modules, and the morphisms are $\ul{G}$-linear maps. We similarly use $\ul{H}$-module and $\ul{H}$-linear.

\subsection{Weights} \label{ss:weights}

A \defn{weight} is a word in the alphabet $\{\wa,\wb\}$. Weights will be the main combinatorial objects we use to describe representations of $G$ and $H$.

Let $\lambda=\lambda_1 \cdots \lambda_n$ be a weight. We define the \defn{length} of $\lambda$, denoted $\ell(\lambda)$, to be $n$. We define\footnote{We regard the symbol $\sigma$ as a general cyclic shift operator, which is why we use the same notation for this action as the one on $\bS^{\{n\}}$.} $\sigma(\lambda)=\lambda_n \lambda_1 \lambda_2 \cdots \lambda_{n-1}$. This construction defines an action of $\bZ/n$ on the set of weights of length $n$. We refer to $\sigma^i(\lambda)$ as the $i$th \defn{cyclic shift} of $\lambda$. We let $\Aut(\lambda) \subset \bZ/n$ be the stabilizer of $\lambda$, i.e., the set of $i \in \bZ/n$ such that $\sigma^i(\lambda)=\lambda$. We put $g(\lambda)=\# \Aut(\lambda)$ and $N(\lambda)=\# [\lambda]$; these numbers multiply to $n$. Note that $\sigma^{N(\lambda)}$ generates $\Aut(\lambda)$ (if $\sigma$ is the generator of $\bZ/n$).


Let $\lambda$ be as above, and assume $n>0$. Define $\gamma_1(\lambda)=\lambda_2 \cdots \lambda_n$, which is a weight of length $n-1$. For $i \in \bZ/n$, define $\gamma_i(\lambda)=\gamma_1(\sigma^{i-1}(\lambda))$. Explicitly, $\gamma_i(\lambda)=\lambda_{i+1} \cdots \lambda_n \lambda_1 \cdots \lambda_{i-1}$. We refer to $\gamma_i(\lambda)$ as the $i$th \defn{cyclic contraction} of $\lambda$. We note that if $\gamma_i(\lambda)=\gamma_j(\lambda)$ then $\sigma^i(\lambda)=\sigma^j(\lambda)$. Indeed, the hypothesis shows that $\sigma^i(\lambda)$ and $\sigma^j(\lambda)$ agree except for perhaps the first letter, but these have to be the same by counting the number of $\wa$'s and $\wb$'s in $\lambda$.

\subsection{Representations of $H$} \label{ss:RepH}

The predecessor to this paper \cite{line} studied the category $\uRep(H)$ in detail. We will use several results from this paper, the most important of which is the description of the simple objects, which we now recall.

A \defn{half-open interval} in $\bR$ is a non-empty interval of the form $(b,a]$ or $[a,b)$ where $a \in \bR$ and $b \in \bR \cup \{\pm \infty\}$. We define the \defn{type} of a half-open interval to be $\wa$ if its right endpoint is included, and $\wb$ if its left endpoint is included. For two intervals $I$ and $J$, we write $I<J$ to mean $x<y$ holds for all $x \in I$ and $y \in J$. We also write $I \ll J$ to mean $\ol{I}<\ol{J}$, where the bar denotes closure. Thus, for instance, $[0,1)<[1,2)$ is true, while $[0,1) \ll [1,2)$ is false.

Let $\bI=(I_1, \ldots, I_n)$ be a tuple of half-open intervals. We assume that $\bI$ is \defn{ordered}, meaning $I_1<\cdots<I_n$. We identify $\bI$ with the cube $I_1 \times \cdots \times I_n$ in $\bR^{(n)}$. We let $\phi_{\bI} \in \cC(\bR^{(n)})$ be the characteristic function of $\bI$. We define the \defn{type} of $\bI$ to be the weight $\lambda_1 \cdots \lambda_n$, where $\lambda_i$ is the type of $I_i$.

For a weight $\lambda$ of length $n$, we define $L_{\lambda}$ to be the $\ul{H}$-submodule of $\cC(\bR^{(n)})$ generated by the functions $\phi_{\bI}$ where $\bI$ is an ordered tuple of type $\lambda$. The module $L_{\lambda}$ is simple \cite[Theorem~4.3(a)]{line}, if $\lambda \ne \mu$ then $L_{\lambda}$ and $L_{\mu}$ are non-isomorphic \cite[Theorem~4.3(b)]{line}, and every simple is isomorphic to some $L_{\lambda}$ \cite[Corollary~4.12]{line}. Note that since $L_{\lambda}$ is simple, it is generated as an $\ul{H}$-module by any non-zero element; in particular, we could use generators $\phi_{\bI}$ where $I_1 \ll \cdots \ll I_n$. The module $L_{\lambda}$ has categorical dimension $(-1)^n$ \cite[Corollary~5.7]{line}.

Suppose that $a \in \bR^{(n)}$. Then the $\ul{H}(a)$-invariants and $H(a)$-invariants in any $\ul{H}$-module are the same \cite[Proposition~11.17]{repst}, and naturally identified with the $\ul{H}(a)$-coinvariants by semi-simplicity. If $\lambda$ has length $m$ then the space $L_{\lambda}^{H(a)}$ has dimension $\binom{m}{n}$ \cite[Corollary~4.11]{line}. This is an important result that we will often use. It is possible to write down an explicit basis for the invariants (see \cite[Corollary~5.5]{line}), but is much easier to describe the co-invariants. For $x \in \bR^{(n)}$, let $\ev_x \colon \cC(\bR^{(n)}) \to k$ be the map $\ev_x(\phi)=\phi(x)$, which is $\ul{H}(x)$-equivariant. The maps $\ev_x$, with $x$ an $m$-element subset of $a$, give a basis for the $\ul{H}(a)$-coinvariants of $L_{\lambda}$ (see \cite[Proposition~4.6]{line}). We note that if $b \in \bS^{\{n+1\}}$ is the point $(a_1, \ldots, a_n, \infty)$ then $H(a)=G(b)$, and so the above results can be phrased in terms of $G(b)$-invariants too.

\subsection{Transpose} \label{ss:transp}

Let $r \colon \bS \to \bS$ be the reflection about the $y$-axis (recall $\bS$ is the unit circle in the plane). Conjugation by $r$ induces a continuous outer automorphism of the group $G$. For a $\ul{G}$-module $M$, we let $M^{\dag}$ be its conjugate under $r$; we call this the \defn{transpose} of $M$. Transpose defines a covariant involutive auto-equivalence
\begin{displaymath}
(-)^{\dag} \colon \uRep(G) \to \uRep(G).
\end{displaymath}
The map $r$ induces an isomorphism $\bS^n \to \bS^n$ of $G$-sets, though it does not map $\bS^{\{n\}}$ into itself for $n \ge 3$. However, letting $\sigma$ be the longest element of the symmetric group $\fS_n$, the composition $r'=\sigma \circ r$ does map $\bS^{\{n\}}$ to itself. Pull-back by $r'$ is an isomorphism $\cC(\bS^{\{n\}})^{\dag} \cong \cC(\bS^{\{n\}})$ of $\ul{G}$-modules.

We take $\infty$ to be the north pole of $\bS$. Thus $r$ fixes $\infty$, and so normalizes $H$. Thus we can define the transpose of an $\ul{H}$-module as well. As above, $r'$ induces an isomorphism $\cC(\bR^{(n)})^{\dag} \cong \cC(\bR^{(n)})$. This isomorphism takes $\phi_{\bI}$ to $\phi_{\bJ}$, where $\bJ=(r(I_n), \ldots, r(I_1))$. If $\bI$ has type $\lambda$ then $\bJ$ has type $\lambda^{\dag}$, where $\lambda^{\dag}$ is the result of reversing $\lambda$ and changing each $\wa$ to $\wb$. In particular, we find $L_{\lambda}^{\dag} \cong L_{\lambda^{\dag}}$, as noted in \cite[Remark~4.17]{line}.

\subsection{Induced modules} \label{ss:ind}

For a weight $\lambda$, let $I_{\lambda}=\Ind_H^G(L_{\lambda})$. These modules will play an important role in our analysis of $\uRep(G)$. We prove a few simple results about them here.

\begin{proposition}
The module $I_{\lambda}$ is both projective and injective in $\uRep(G)$.
\end{proposition}

\begin{proof}
Since restriction is exact and induction is both its left and right adjoint, it follows that induction preserves injective and projective objects. Since $\uRep(H)$ is semi-simple, all objects are injective and projective.
\end{proof}

We next determine the decomposition of $I_{\lambda}$ into simple $\ul{H}$-modules. For this, we introduce some notation. Let $\lambda$ be a weight of length $n \ge 0$. Put
\begin{displaymath}
A_1(\lambda) = \bigoplus_{1 \le i \le n} L_{\sigma^i(\lambda)}, \qquad
A_2(\lambda) = \bigoplus_{1 \le i \le n} L_{\gamma_i(\lambda)}, \qquad
A(\lambda) = A_1(\lambda) \oplus A_2(\lambda).
\end{displaymath}
Note that for $\lambda = \emptyset$ we have empty sums, and so $A(\emptyset)$ is the zero module. Also, put $\lambda^+=\lambda \wa$ and $\lambda^-=\lambda \wb$.

\begin{proposition} \label{prop:ind}
We have an isomorphism of $\ul{H}$-modules
\begin{displaymath}
I_{\lambda} = A(\lambda) \oplus A(\lambda^+) \oplus A(\lambda^-).
\end{displaymath}
\end{proposition}

\begin{proof}
By Mackey decomposition (Proposition~\ref{prop:mackey}),
\begin{displaymath}
\Res^G_H(\Ind_H^G(L_\lambda)) = \bigoplus_{g \in H \backslash G/H} \Ind^H_{H^g \cap H}(\Res^{H^g}_{H^g \cap H}(L_\lambda^g)).
\end{displaymath}

We first analyze the double cosets and the corresponding subgroups. Since $G/H = \bS$, the double cosets $H \backslash G /H$ are identified with the $H$ orbits on $\bS$, of which there are exactly two. For the trivial double coset, $H^e \cap H = H$, and $L_\lambda^g = L_\lambda$. So this double coset contributes exactly $L_\lambda$. For the non-trivial double coset we pick a representative $g$ sending $\infty$ to $0$, so $H^g \cap H = H(0)$. Since $\bR - \{0\} = \bR_{< 0} \cup \bR_{> 0}$, we have that $H(0) \cong H \times H$, and for clarity we will denote these subgroups $H_{(-\infty,0)}$ and $H_{(0,+\infty)}$. Now we come to the key point, which is that conjugation by $g$ interchanges the two subgroups $H_{(-\infty,0)}$ and $H_{(0,+\infty)}$.

Before considering representations contributed by the non-trivial double coset, we recall the branching rules between $H$ and $H(0)$. For this we will need a piece of notation: for a weight $\lambda=\lambda_1 \cdots \lambda_n$ of length $n$, we let $\lambda[i,j]$ denote the substring of $\lambda$ between indices $i$ and $j$ (inclusively), i.e., $\lambda_i \cdots \lambda_j$. As with intervals, we use parentheses to exclude the edge values, e.g., $\lambda[i,j)=\lambda_i \cdots \lambda_{j-1}$. Now, \cite[Theorem~6.1]{line} states
\begin{displaymath}
\Res_{H(0)}^H L_\lambda \cong \bigoplus_{i=0}^n \big( L_{\lambda[1,i]} \boxtimes L_{\lambda(i,n]} \big) \oplus \bigoplus_{i=1}^n \big( L_{\lambda[1,i)} \boxtimes L_{\lambda(i,n]} \big)
\end{displaymath}
and \cite[Theorem~6.9]{line} states
\begin{displaymath}
\Ind_{H(0)}^H \left(L_{\lambda_1} \boxtimes L_{\lambda_2}\right) \cong L_{\lambda_1 \wa \lambda_2} \oplus L_{\lambda_1 \wb \lambda_2} \oplus L_{\lambda_1 \lambda_2}
\end{displaymath}

Keeping in mind that conjugation by $g$ interchanges the two subgroups $H_{(-\infty,0)}$ and $H_{(0,+\infty)}$, we have
\begin{displaymath}
\Res^{H^g}_{H(0)}(L_\lambda^g) \cong \bigoplus_{i=0}^n  \big( L_{\lambda(i,n]} \boxtimes L_{\lambda[1,i]} \big) \oplus \bigoplus_{i=1}^n \big( L_{\lambda(i,n]} \boxtimes L_{\lambda[1,i)} \big),
\end{displaymath}
and so, 
\begin{align*}
\Ind_{H(0)}^H \Res^{H^g}_{H(0)}(L_\lambda^g) &\cong
\bigoplus_{i=0}^n  \Ind_{H(0)}^H \left(L_{\lambda(i,n]} \boxtimes L_{\lambda[1,i]}\right)  \oplus \bigoplus_{i=1}^n \Ind_{H(0)}^H \left(L_{\lambda(i,n]} \boxtimes L_{\lambda[1,i)} \right)\\
& \cong  \bigoplus_{i=0}^n  L_{\lambda(i,n] \wa \lambda[1,i]} \oplus
\bigoplus_{i=0}^n  L_{\lambda(i,n] \wb \lambda[1,i]} \oplus
\bigoplus_{i=0}^n  L_{\lambda(i,n] \lambda[1,i]}  \\ 
&\quad \oplus \bigoplus_{i=1}^n  L_{\lambda(i,n] \wa \lambda[1,i)}
\oplus \bigoplus_{i=1}^n  L_{\lambda(i,n] \wb \lambda[1,i)}
\oplus \bigoplus_{i=1}^n  L_{\lambda(i,n] \lambda[1,i)}
\end{align*}

Now note that these summands are exactly $L_{\sigma^i(\lambda \wa)}$, $L_{\sigma^i(\lambda \wb)}$, $L_{\sigma^i(\lambda)}$, $L_{\gamma_i(\lambda \wa)}$, $L_{\gamma_i(\lambda \wb)}$, and $L_{\gamma_i(\lambda)}$, respectively. This gives exactly one copy of each cyclic rotation and each cyclic contraction of $\lambda \wa$, $\lambda \wb$, and $\lambda$, except that we miss $L_{\gamma_{n+1}(\lambda \wa)} = L_\lambda$ and $L_{\gamma_{n+1}(\lambda \wb)} = L_\lambda$, and have an extra copy of $L_{\sigma^n(\lambda)} = L_\lambda$. Finally, including the additional copy of $L_\lambda$ from the trivial double coset, we get exactly $A(\lambda) \oplus A(\lambda^+) \oplus A(\lambda^-)$.
\end{proof}

We finally observe that we can decompose Schwartz space $\cC(\bS^{\{n\}})$ into induced modules.

\begin{proposition} \label{prop:schwartz-ind}
For $n \ge 0$, we have an isomorphism of $\ul{G}$-modules
\begin{displaymath}
\cC(\bS^{\{n+1\}}) = \bigoplus_{\ell(\lambda) \le n} I_{\lambda}^{\oplus m_{\lambda}}, \qquad m_{\lambda} = \binom{n}{\ell(\lambda)}.
\end{displaymath}
\end{proposition}

\begin{proof}
We have an isomorphism of $G$-sets $\bS^{\{n+1\}} \cong G \times^H \bR^{(n)}$, and so a $\ul{G}$-isomorphism
\begin{displaymath}
\cC(\bS^{\{n+1\}}) = \Ind_H^G(\cC(\bR^{(n)}))
\end{displaymath}
by \cite[Proposition~2.13]{line}. We have an $\ul{H}$-isomorphism
\begin{displaymath}
\cC(\bR^{(n)}) = \bigoplus_{\ell(\lambda) \le n} L_{\lambda}^{\oplus m_{\lambda}}
\end{displaymath}
by \cite[Theorem~4.7]{line}. Thus the result follows.
\end{proof}

\section{Standard modules} \label{s:std}

\subsection{Construction of standard modules}

For distinct $a,b \in \bS$, we let $[a,b]$ denote the interval consisting of points $x \in \bS$ such that $a \le x \le b$. We use the usual parenthesis notation to omit endpoints. A \defn{half-open interval} is one of the form $(a,b]$ or $[a,b)$, with $a \ne b$. We define the \defn{type} of a half-open interval to be $\wa$ in the $(a,b]$ case and $\wb$ in the $[a,b)$ case. Suppose that $I_1, \ldots, I_n$ are subsets of $\bS$. We write $I_1<\cdots<I_n<I_1$ to mean that $x_1<\cdots<x_n<x_1$ for all $x_i \in I_i$, and we write $I_1 \ll \cdots \ll I_n \ll I_1$ to mean that $\ol{I}_1<\cdots<\ol{I}_n<\ol{I}_1$, where the bar denotes topological closure.

Let $\bI=(I_1, \ldots, I_n)$ be a tuple of half-open intervals in $\bS$. We say that $\bI$ is \defn{ordered} if $I_1<\cdots<I_n<I_1$, and \defn{strictly ordered} if $I_1 \ll \cdots \ll I_n \ll I_1$. Assuming $\bI$ is ordered, we identify $\bI$ with the cube $I_1 \times \cdots \times I_n$ in $\bS^{\{n\}}$, and write $\psi_{\bI} \in \cC(\bS^{\{n\}})$ for its characteristic function. We define the \defn{type} of $\bI$ to be the word $\lambda_1 \cdots \lambda_n$ where $\lambda_i$ is the type of $I_i$. We can now introduce an important class of modules.

\begin{definition} \label{defn:std}
For a word $\lambda$ of length $n>0$, we define the \defn{standard module} $\Delta_{\lambda}$ to be the $\ul{G}$-submodule of $\cC(\bS^{\{n\}})$ generated by the functions $\psi_{\bI}$, with $\bI$ a strictly ordered tuple of type $\lambda$.
\end{definition}

Recall that $\bZ/n=\langle \sigma \rangle$ acts on $\bS^{\{n\}}$ by cyclicly permuting coordinates. We have $\sigma^*(\psi_{\bI})=\psi_{\sigma^{-1} \bI}$. We thus see that $\sigma^*(\Delta_{\lambda})=\Delta_{\sigma^{-1}\lambda}$. In particular, $\Aut(\lambda)$ acts on $\Delta_{\lambda}$. For a $g(\lambda)$-root of unity $\zeta$, let $\Delta_{\lambda,\zeta}$ be the $\zeta$-eigenspace of the action of $\sigma^{N(\lambda)}$ on $\Delta_{\lambda}$. If $\mu$ is a cyclic shift of $\lambda$ then $\Delta_{\lambda}$ and $\Delta_{\mu}$ are isomorphic, as are $\Delta_{\lambda,\zeta}$ and $\Delta_{\mu,\zeta}$, via an appropriate power of $\sigma^*$.

\begin{remark}
In our previous paper \cite{line}, we worked over an arbitrary field. In this paper, we only work with an algebraically closed field $k$ of characteristic~0. Without this assumption we would need to take more care in analyzing the representation theory of $\bZ/n$ in the construction of $\Delta_{\lambda, \zeta}$ and throughout the rest of the paper. With more work, one should be able to extend many of our results to more general fields.
\end{remark}


We defined $\Delta_{\lambda}$ using strictly ordered tuples, since they form a single $G$-orbit. The following proposition shows we get the same result using all ordered tuples.

\begin{proposition} \label{prop:merely-ordered}
If $\bI$ is ordered and has type $\lambda$, then $\psi_\bI \in \Delta_\lambda$
\end{proposition}

\begin{proof}
Let $I_i^{\circ}$ be the interior of $I_i$, and let $\bI^{\circ}$ be their product. For $t \in I_i^{\circ}$, let $J_i(t)$ be the half-open interval of type $\lambda_i$ with the same closed endpoint as $I_i$, and with open endpoint $t$. We claim that
\begin{displaymath}
\psi_{\bI} = (-1)^n \int_{\bI^{\circ}}\psi_{\bJ(t)} \; dt.
\end{displaymath}
Indeed, it suffices to check that the values agree for each point $a \in \bS^{\{n\}}$. If $a \not\in \bI$ then $a \not\in \bJ(t)$ for any $t$, and so the integral is zero. Now suppose that $a \in \bI$. Then the value of the integral, evaluated at $a$, is the volume of the set $\{ t \in \bI^{\circ} \mid a \in \bJ(t) \}$. This is a product of non-empty open intervals, and thus has volume $(-1)^n$. Since $\bJ(t)$ belongs to $\Delta_{\lambda}$ for all $t \in \bI^{\circ}$, the integral belongs to $\Delta_{\lambda}$, and thus so does $\psi_{\bI}$.
\end{proof}

\begin{example} \label{ex:std}
The module $\Delta_{\wb}$ is generated by functions of the form $\psi_{[a,b)}$ with $a,b \in \bS$ distinct. From the expression
\begin{displaymath}
1=\psi_{[\infty,0)}+\psi_{[0,\infty)},
\end{displaymath}
we see that the constant function~1 belongs to $\Delta_{\wb}$. This generates a trivial subrepresentation of $\Delta_{\wb}$, and it is clearly the unique trivial subrepresentation (as a $G$-invariant function on $\bS$ is constant). We will see in Theorem~\ref{thm:std} below that the quotient is simple.
\end{example}

\begin{remark} \label{rmk:std-transp}
The discussion in \S \ref{ss:transp} shows that $\Delta_{\lambda}^{\dag} \cong \Delta_{\lambda^{\dag}}$, and also that $\Delta_{\lambda,\zeta}^{\dag} \cong \Delta_{\lambda^{\dag},\zeta^{-1}}$. Note that the isomorphism $\cC(\bS^{\{n\}})^{\dag} \cong \cC(\bS^{\{n\}})$ intertwines the actions of $\sigma$ and $\sigma^{-1}$; this is why we get $\zeta^{-1}$.
\end{remark}

\subsection{Statement of results}

Fix a weight $\lambda$ of length $n>0$, and a $g(\lambda)$-root of unity $\zeta$. Recall from \S \ref{ss:ind} that we have defined $\ul{H}$-modules
\begin{displaymath}
A_1(\lambda) = \bigoplus_{1 \le i \le n} L_{\sigma^i(\lambda)}, \qquad
A_2(\lambda) = \bigoplus_{1 \le i \le n} L_{\gamma_i(\lambda)}, \qquad
A(\lambda) = A_1(\lambda) \oplus A_2(\lambda).
\end{displaymath}
We note that $\Aut(\lambda)$ permutes the summands in $A_1(\lambda)$ and $A_2(\lambda)$, and so the above modules can be viewed as $\ul{H} \times \Aut(\lambda)$ modules. We put
\begin{displaymath}
\ol{A}_1(\lambda) = \bigoplus_{1 \le i \le N(\lambda)} L_{\sigma^i(\lambda)}, \qquad
\ol{A}_2(\lambda) = \bigoplus_{1 \le i \le N(\lambda)} L_{\gamma_i(\lambda)}, \qquad
\ol{A}(\lambda) = \ol{A}_1(\lambda) \oplus \ol{A}_2(\lambda).
\end{displaymath}
Thus $A(\lambda) = \ol{A}(\lambda) \boxtimes k[\Aut(\lambda)]$ as a $\ul{H} \times \Aut(\lambda)$ module. The following is the main theorem of this section.

\begin{theorem} \label{thm:std}
Let $\lambda$ and $\zeta$ be as above.
\begin{enumerate}
\item We have $\Delta_{\lambda} \cong A(\lambda)$ as an $\ul{H} \times \Aut(\lambda)$ module.
\item We have $\Delta_{\lambda,\zeta} \cong \ol{A}(\lambda)$ as an $\ul{H}$-module.
\item The $\ul{G}$-module $\Delta_{\lambda,\zeta}$ has length at most two. If it is not simple then $\ol{A}_2(\lambda)$ is its unique proper non-zero submodule.
\item Let $b \in \bS^{\{n\}}$ and let $\tau$ be the standard generator of $G[b]/G(b) \cong \bZ/n$. Then $\Delta^{G(b)}_{\lambda,\zeta}$ has dimension $N(\lambda)$, and $\tau^{N(\lambda)}$ acts on it by $\zeta$.
\end{enumerate}
\end{theorem}

We offer a word of clarification on (c). Technically speaking, $\ol{A}_2(\lambda)$ is not a subspace of $\Delta_{\lambda,\zeta}$. However, there is a unique copy of $\ol{A}_2(\lambda)$ in $\Delta_{\lambda,\zeta}$, namely, the sum of the $L_{\gamma_i(\lambda)}$ isotypic pieces. This is what we mean in (c).

We observe a few corollaries of the theorem.

\begin{corollary}
We have $\End_{\ul{G}}(\Delta_{\lambda,\zeta})=k$. In particular, $\Delta_{\lambda,\zeta}$ is indecomposable.
\end{corollary}

\begin{proof}
The $\ul{H}$-module $L_{\lambda}$ appears with multiplicity one in $\Delta_{\lambda,\zeta}$. We therefore have a map
\begin{displaymath}
\End_{\ul{G}}(\Delta_{\lambda,\zeta}) \to \End_{\ul{H}}(L_{\lambda}) = k.
\end{displaymath}
We claim that this map is injective. Indeed, suppose that $f$ is in the kernel, i.e., $f$ is a $\ul{G}$-endomorphism of $\Delta_{\lambda,\zeta}$ that kills $L_{\lambda}$. Then $\ker(f)$ is a $\ul{G}$-submodule of $\Delta_{\lambda,\zeta}$ containing $L_{\lambda}$, and is thus all of $\Delta_{\lambda,\zeta}$ by Theorem~\ref{thm:std}(c). Hence $f=0$, which proves the claim, and completes the proof.
\end{proof}

\begin{corollary} \label{cor:std-isom}
The $\ul{G}$-modules $\Delta_{\lambda,\zeta}$ and $\Delta_{\mu,\omega}$ are isomorphic if and only if $\omega=\zeta$ and $\mu$ is a cyclic shift of $\lambda$ (i.e., $\mu=\sigma^i(\lambda)$ for some $i$).
\end{corollary}

\begin{proof}
We have already seen that $\Delta_{\lambda,\zeta}$ and $\Delta_{\sigma^i(\lambda),\zeta}$ are isomorphic. Suppose now that we have an isomorphism $f \colon \Delta_{\lambda,\zeta} \to \Delta_{\mu,\omega}$. Let $n=\ell(\lambda)$ and $m=\ell(\mu)$. Since $\Delta_{\lambda,\zeta}$ only has simple $\ul{H}$-modules of lengths $n$ and $n-1$, and $\Delta_{\mu,\omega}$ only has ones of lengths $m$ and $m-1$, we must have $m=n$. Looking at the $\ul{H}$-submodules of length $n$, we find $A_1(\lambda) \cong A_1(\mu)$ as $\ul{H}$-modules, which implies that $\mu$ is a cyclic shift of $\lambda$. Let $b \in \bS^{\{n\}}$, and let $\tau$ be the generator of $G[b]/G(b)$. We have an isomorphism $f \colon \Delta_{\lambda,\zeta}^{G(b)} \to \Delta_{\mu,\omega}^{G(b)}$ that commutes with the action of $\tau$. Since $\tau^{N(\lambda)}$ acts by $\zeta$ on the source and $\omega$ on the target, we conclude that $\omega=\zeta$.
\end{proof}

\begin{corollary}
$\Delta_{\lambda,\zeta}$ has categorical dimension~0.
\end{corollary}

\begin{proof}
The modules $\ol{A}_1(\lambda)$ and $\ol{A}_2(\lambda)$ have dimension of absolute value $N(\lambda)$, but with different signs, by \cite[Corollary~5.7]{line}.
\end{proof}

Since $\Delta_{\lambda,\zeta}$ has a unique maximal proper submodule (possibly zero), it has a unique simple quotient. We give it a name:

\begin{definition} \label{defn:M}
For a non-empty weight $\lambda$ and a $g(\lambda)$ root of unity $\zeta$, we define $M_{\lambda,\zeta}$ to be the unique simple quotient of $\Delta_{\lambda,\zeta}$.
\end{definition}

We note that $M_{\lambda,\zeta}$ contains $L_{\lambda}$ with multiplicity one, and that all simple $\ul{H}$-modules appearing in $M_{\lambda,\zeta}$ have length $n$ or $n-1$ where $n=\ell(\lambda)$. It follows from Remark~\ref{rmk:std-transp} that $M^{\dag}_{\lambda,\zeta} \cong M_{\lambda^{\dag}, \zeta^{-1}}$.

\subsection{Set-up} \label{ss:std-setup}

The proof of Theorem~\ref{thm:std} will take the remainder of \S \ref{s:std}. We now set-up some notation and make a few initial observations.

Fix a weight $\lambda$ of length $n>0$, and a $g(\lambda)$-root of unity $\zeta$. For an $\ul{H}$-module $X$, we let $X[r]$ be the sum of the $L_{\mu}$-isotypic components of $X$ with $\ell(\mu)=r$. We also use this notation when $X$ is a $\ul{G}$-module, by simply restricting to $\ul{H}$. We note that $\Delta_{\lambda}[r]$ is an $\ul{H} \times \Aut(\lambda)$ module. In this notation, Theorem~\ref{thm:std}(a) states that $\Delta_{\lambda}[n] \cong A_1(\lambda)$ and $\Delta_{\lambda}[n-1] \cong A_2(\lambda)$, as $\ul{H} \times \Aut(\lambda)$ modules, and also $\Delta_{\lambda}[r]=0$ for $r \ne n, n-1$.

Let $\alpha_0 \colon \bR^{(n)} \to \bS^{\{n\}}$ be the standard inclusion, and let $\beta_0 \colon \bR^{(n-1)} \to \bS^{\{n\}}$ be defined by
\begin{displaymath}
\beta_0(x_1, \ldots, x_{n-1}) = (x_1, \ldots, x_{n-1}, \infty).
\end{displaymath}
For $i \in \bZ/n$, let $\alpha_i=\sigma^i \alpha_0$ and $\beta_i=\sigma^i \beta_0$. These maps are $H$-equivariant. We let
\begin{displaymath}
\alpha_* \colon \cC(\bR^{(n)})^{\oplus n} \to \cC(\bS^{\{n\}}), \qquad
\beta_* \colon \cC(\bR^{(n-1)})^{\oplus n} \to \cC(\bS^{\{n\}})
\end{displaymath}
be the direct sum of the $(\alpha_i)_*$ and $(\beta_i)_*$, and similarly define $\alpha^*$ and $\beta^*$. These maps are $\ul{H} \times \bZ/n$ equivariant, where $\bZ/n$ permutes the summands in the domain and acts by shifting on the target. We note that $\alpha_i^*$ simply restricts a function on $\bS^{\{n\}}$ to the $i$th copy of $\bR^{(n)}$, while $(\alpha_i)_*$ extends a function on this $\bR^{(n)}$ to all of $\bS^{\{n\}}$ by zero. The map
\begin{displaymath}
\big( \coprod_{i \in \bZ/n} \bR^{(n)} \big) \amalg \big( \coprod_{i \in \bZ/n} \bR^{(n-1)} \big) \to \bS^{\{n\}}
\end{displaymath}
defined by the $\alpha_i$'s and $\beta_i$'s is an isomorphism of $H$-sets. It follows that the map
\begin{displaymath}
\alpha_* \oplus \beta_* \colon \cC(\bR^{(n)})^{\oplus n} \oplus \cC(\bR^{(n-1)})^{\oplus n} \to \cC(\bS^{\{n\}
})
\end{displaymath}
is an isomorphism of $\ul{H}$-modules; the inverse map is $\alpha^* \oplus \beta^*$. 
We have $\cC(\bR^{(n)})[r]=\cC(\bR^{(n-1)})[r]=0$ for $r>n$ \cite[Theorem~4.7]{line}, and so $\Delta_{\lambda}[r]=0$ for $r>n$ as well.

\subsection{The length $n$ piece}

We now examine $\Delta_{\lambda}[n]$.

\begin{lemma} \label{lem:std-1}
For $i \in \bZ/n$, we have $(\alpha_{-i})_*(L_{\sigma^i(\lambda)}) \subset \Delta_{\lambda}$, and any non-zero element of this space generates $\Delta_{\lambda}$ as a $\ul{G}$-module.
\end{lemma}

\begin{proof}
It suffices to treat the $i=0$ case. Let $\bI=(I_1,\ldots,I_n)$ be a tuple of bounded half-open intervals in $\bR$ of type $\lambda$, with $I_1 \ll \cdots \ll I_n$. We can also regard $\bI$ as a tuple of intervals in $\bS$. As such, it is strictly ordered of type $\lambda$; here it is important that the $I_i$'s are bounded. It follows that $(\alpha_i)_*(\phi_{\bI})=\psi_{\bI}$ is a generator for $\Delta_{\lambda}$. Since $\phi_{\bI}$ generates $L_{\lambda}$ as an $\ul{H}$-module and $\alpha_i$ is $\ul{H}$-equivariant, it follows that $(\alpha_i)_*$ carries $L_{\lambda}$ into $\Delta_{\lambda}$. Any non-zero element of the image will generate the entire image as an $\ul{H}$-module (since $L_{\lambda}$ is irreducible), and thus generate $\psi_{\bI}$, which in turn generates $\Delta_{\lambda}$ as a $\ul{G}$-module.
\end{proof}

Let $Q = \bigoplus_{i<n} \subset \cC(\bR^{(n)})[i]$ be the sum of all $\ul{H}$-simple submodules of length $<n$.

\begin{lemma} \label{lem:std-2}
For $i \in \bZ/n$ we have $\alpha_{-i}^*(\Delta_{\lambda}) \subset L_{\sigma^i(\lambda)}+Q$.
\end{lemma}

\begin{proof}
It suffices to treat the $i=0$ case, so we assume this in what follows. It suffices to show that $\alpha_0^*(\psi_{\bI})$ is contained in $L_{\lambda}+Q$ for each generator $\psi_{\bI}$ of $\Delta_{\lambda}$. Note that if $\infty$ belongs to the interior of some interval in $\bI$ then we can write $\psi_{\bI}=\psi_{\bJ}+\psi_{\bK}$, where $\psi_{\bJ}$ and $\psi_{\bK}$ are generators of $\Delta_{\lambda}$ and $\infty$ is not in the interior of any interval in $\bJ$ or $\bK$. Thus we assume $\infty$ is not in the interior of any interval in $\bI$ in what follows.

We consider four cases:
\begin{enumerate}
\item We have $I_n<\infty<I_1$.
\item $\infty$ is the included left endpoint of $I_1$.
\item $\infty$ is the included right endpoint of $I_n$.
\item All other cases.
\end{enumerate}
In case (a), $\alpha_0^*(\psi_{\bI})$ is one of the generators of $L_{\lambda}$, and so the statement holds. In case (d), we have $\alpha_0^*(\phi_{\bI})=0$. Cases (b) and (c) are essentially the same, and we treat (b).

Assume we are in case (b); in particular, $\lambda_1=\wb$. We regard $\psi_{\bI}$ as an element of $\cC(\bR^{(n)})$; note that $I_1$ was of the form $[-\infty,a)$, but (now that we're ignoring $\infty$) is the open interval $(\infty,a)$. To show that $\psi_{\bI}$ belongs to $L_{\lambda}+Q$, it suffices (by \cite[\S 4.5]{line}) to show that it is orthogonal to the spaces $L_{\mu}$ where $\ell(\mu)=n$ and $\mu \ne \lambda^{\vee}$. Thus let $\mu$ be given, and let $\bJ$ be a tuple of type $\bJ$; we show that $\langle \psi_{\bI}, \phi_{\bJ} \rangle=0$.

This pairing is equal to the product of $\vol(I_i \cap J_i)$ for $1 \le i \le n$, so we must show one of these vanishes. Since $\mu \ne \lambda^{\vee}$, there is some $i$ such that $\mu_i=\lambda_i$. If $i>1$ then this means that $I_i \cap J_i$ is a half-open interval (or empty), and thus has volume~0. In fact, the same is true for $i=1$. Indeed, we have $\mu_1=\lambda_1=\wb$, and so $J_1=[c,d)$ for some $c<d$. If $c>a$ then $I_1 \cap J_1=\emptyset$, while otherwise the intersection is half-open of type $\wb$ with left endpoint $c$.
\end{proof}

\begin{example}
Here is the main example to keep in mind for Lemma~\ref{lem:std-2}. The module $\Delta_{\wb}$ contains the function $\psi_{[\infty,0)} \in \cC(\bS)$. When we apply $\alpha^*_0$ to this, we get the function $\psi_{(-\infty,0)} \in \cC(\bR)$. This belongs to (and generates) $L_{\wb}+L_{\emptyset}$. To see this explicitly, simply note that $\psi_{(-\infty,0)}=\psi_{\bR}-\psi_{[0,\infty)}$ and the two terms belong to $L_{\emptyset}$ and $L_{\wb}$.
\end{example}

\begin{lemma} \label{lem:std-3}
We have $\Delta_{\lambda}[n] \cong A_1(\lambda)$ as $\ul{H} \times \Aut(\lambda)$ modules.
\end{lemma}

\begin{proof}
 Since $\cC(\bR^{(n-1)})[n]=0$, it follows that $\alpha^*$ is injective on $\Delta_{\lambda}[n]$. Thus $\Delta_{\lambda}[n]$ has $\ul{H}$-length at most $n$ by Lemma~\ref{lem:std-2}. Now, regard $A_1(\lambda)$ as a submodule of $\cC(\bR^{(n)})^{\oplus n}$ by including the $i$th summand of the former into the $i$th summand of the latter. Then $\alpha_*$ defines an injective map $A_1(\lambda) \to \Delta(\lambda)[n]$ by Lemma~\ref{lem:std-1}, which is clearly $\ul{H} \times \Aut(\lambda)$ equivariant. Since $A_1(\lambda)$ has $\ul{H}$-length $n$, this map is an isomorphism.
\end{proof}

\begin{lemma} \label{lem:std-4}
We have $\Delta_{\lambda,\zeta}[n] \cong \ol{A}_1(\lambda)$ as $\ul{H}$-modules, and any non-zero element of this space generates $\Delta_{\lambda,\zeta}$ as a $\ul{G}$-module.
\end{lemma}

\begin{proof}
The first statement follows from taking $\zeta$-eigenspaces in Lemma~\ref{lem:std-3}; recall that $A_1(\lambda) \cong \ol{A}_1(\lambda) \boxtimes k[\Aut(\lambda)]$. To prove the second, it suffices to show that the unique copy of $L_{\sigma^i(\lambda)}$ in $\Delta_{\lambda,\zeta}$ generates it as a $\ul{G}$-module, for any $i \in \bZ/n$. It suffices to treat the $i=0$ case, so we assume this. Consider the composition $L_{\lambda} \to \Delta_{\lambda} \to \Delta_{\lambda,\zeta}$, where the first map is $\alpha_0^*$ and the second is the natural quotient map. The composition is non-zero: indeed, the multiplicity space of $L_{\lambda}$ in $\Delta_{\lambda}$ is identified with $k[\Aut(\lambda)]$, and $\alpha_0^*$ (regarded as an element of this space) is identified with $1 \in \Aut(\lambda)$, which has non-zero projection to the $\Aut(\lambda)$-invariants of the regular representation. Since the image of $L_{\lambda}$ in $\Delta_{\lambda}$ generates it as a $\ul{G}$-module (Lemma~\ref{lem:std-1}), the same is true for its image in $\Delta_{\lambda,\zeta}$.
\end{proof}

\subsection{The length $n-1$ piece}

So far, we have shown that $\Delta_{\lambda}[r]$ vanishes for $r>n$ and is what we want when $r=n$. The following lemma handles the $r<n-1$ cases, and gives some information in the $r=n-1$ case.

\begin{lemma} \label{lem:std-5}
We have $\Delta_{\lambda}[r]=0$ for $r<n-1$. Additionally, $\Delta_{\lambda}[n-1]$ is isomorphic to a quotient of $A_2(\lambda)$.
\end{lemma}

\begin{proof}
By Lemma~\ref{lem:std-1}, $\Delta_{\lambda}$ is generated as a $\ul{G}$-module by a copy of $L_{\lambda}$, and so there is a surjection $I_{\lambda} \to \Delta_{\lambda}$ of $\ul{G}$-modules. Since $I_{\lambda}[r]=0$ for $r<n-1$ and $I_{\lambda}[n-1]=A_2(\lambda)$ by Proposition~\ref{prop:ind}, the result follows.
\end{proof}

We next turn our attention to $\Delta_{\lambda}[n-1]$. We regard $A_2(\lambda)$ as a subspace of $\cC(\bR^{(n-1)})$ by identifying the $i$th summand of the former with a subspace of the $i$th summand of the latter.

\begin{lemma} \label{lem:std-6}
The map $\beta^*$ defines an isomorphism $\Delta_{\lambda}[n-1] \to A_2(\lambda)$ of $\ul{H} \times \Aut(\lambda)$-modules.
\end{lemma}

\begin{proof}
Consider a generator $\psi_{\bI}$ of $\Delta_{\lambda}$, where $\infty \in I_n$. Let $\bJ=(I_1, \ldots, I_{n-1})$, which is an ordered tuple of half-open intervals in $\bR$ of type $\gamma_0(\lambda)$. We have $\beta_0^*(\psi_{\bI})=\phi_{\bJ}$, and $\beta_i^*(\psi_{\bI})=0$ for $i \ne 0$. It follows that $\beta^*(\Delta_{\lambda})$ contains the 0th summand of $A_2(\lambda)$ (as this summand is $\ul{H}$-irreducible). By symmetry, it contains each summand of $A_2(\lambda)$, and so it contains all of $A_2(\lambda)$. Since $\cC(\bR^{(n-1)})[r]=0$ for $r \ge n$ and $\Delta_{\lambda}[r]=0$ for $r<n-1$, it follows that $\beta^*(\Delta_{\lambda})=\beta^*(\Delta_{\lambda}[n-1])$. Since $\Delta_{\lambda}[n-1]$ has length $\le n$ by Lemma~\ref{lem:std-5} and $A_2(\lambda)$ has length $n$, it follows that $\beta^*$ maps $\Delta_{\lambda}[n-1]$ isomorphically to $A_2(\lambda)$. As this map is $\ul{H} \times \Aut(\lambda)$ equivariant, the result follows.
\end{proof}

Since $A_2(\lambda) \cong \ol{A}_2(\lambda) \boxtimes k[\Aut(\lambda)]$, its $\zeta$-eigenspace is $\ol{A}_2(\lambda)$. We thus have $\Delta_{\lambda,\zeta}[n-1] \cong \ol{A}_2(\lambda)$, which completes the proof of Theorem~\ref{thm:std}(a,b).

\subsection{Action of normalizers}

We now prove Theorem~\ref{thm:std}(c,d). Fix $b \in \bS^{\{n\}}$; without loss of generality, we take $b_n=\infty$. Let $a=(b_1,\ldots,b_{n-1}) \in \bR^{(n-1)}$, so that $G(b)=H(a)$. If $\mu$ is weight of length $\ge n$ then $L_{\mu}^{G(b)}=0$, while if $\mu$ has length $n-1$ then $L_{\mu}^{G(b)}$ is one-dimensional \cite[Corollary~4.11]{line}. For $1 \le i \le n$, let $Y_i=L_{\gamma_i(\lambda)}^{G(b)}$, regarded as a one-dimensional subspace of $A_2(\lambda)$, and let $X_i \subset \Delta_{\lambda}$ be its inverse image under $\beta^*$. Thus $\Delta_{\lambda}^{G(b)}$ is the direct sum of the lines $X_i$ for $i \in \bZ/n$. Let $\tau$ be the standard generator of $G[b]/G(b) \cong \bZ/n$. Recall that $\sigma^{N(\lambda)}$ is the generator of $\Aut(\lambda) \subset \bZ/n$.

\begin{lemma} \label{lem:std-7}
We have $\tau(X_i)=X_{i+1}$ and $\sigma^{N(\lambda)}(X_i)=X_{i+N(\lambda)}$ for $i \in \bZ/n$.
\end{lemma}

\begin{proof}
Let $\ev_b \colon \cC(\bS^{\{n\}}) \to k$ be the evaluation at $b$ map, defined by $\ev_b(\phi)=\phi(b)$. This map is $\ul{G}(b)$-equivariant. Similarly, let $\ev_a \colon \cC(\bR^{(n-1)}) \to k$ be the evaluation at $a$ map. It is clear that $\ev_b = \ev_a \circ \beta_0^*$. Now, let $\ev_a^i \colon \cC(\bR^{(n-1)})^{\oplus n} \to k$ be the map $\ev_a$ on the $i$th summand and~0 on the remaining summands. We thus have $\ev_b=\ev_a^0 \circ \beta^*$, and so $\ev_{\sigma^i(b)} = \ev_a^i \circ \beta^*$ for $i \in \bZ/n$.

Now, any non-zero $H(a)$-invariant in $L_{\gamma_i(\lambda)} \subset \cC(\bR^{(n-1)})$ has non-zero image under $\ev_a$. (This follows from semi-simplicity, or the explicit description of invariants in \cite[\S 5]{line}.) It follows that $Y_i$ can be described as the joint kernels of $\ev^j_a$, for $j \ne i$, on $A_2(\lambda)^{H(a)}$. It thus follows that $X_i$ is the joint kernels of $\ev_{\sigma^j(b)}$, for $j \ne i$, on $\Delta_{\lambda}^{G(b)}$. Since $\tau(\sigma^i(b))=\sigma^{i+1}(b)$, the result thus follows.
\end{proof}

Let $\ol{X}_i \subset \Delta_{\lambda,\zeta}^{G(b)}$, for $i \in \bZ/N(\lambda)$, be the intersection of $\Delta_{\lambda,\zeta}^{G(b)}$ with the $L_{\gamma_i(\lambda)}$-isotypic component. It is clear from the above discussion that $\ol{X}_i$ is one-dimensional and that $\Delta_{\lambda,\zeta}^{G(b)}$ is the direct sum of these spaces.

\begin{lemma} \label{lem:std-8}
We have $\tau(\ol{X}_i)=\ol{X}_{i+1}$ for $i \in \bZ/N(\lambda)$. Also, $\tau^{N(\lambda)}$ acts by $\zeta$ on $\Delta_{\lambda,\zeta}^{G(b)}$.
\end{lemma}

\begin{proof}
For a weight $\mu$ and a $\ul{G}$-module $M$, write $M[\mu]$ for the $L_{\mu}$ isotypic piece of $M$. We have
\begin{displaymath}
\Delta_{\lambda}^{G(b)} \cap \Delta_{\lambda}[\gamma_i(\lambda)] = \bigoplus_{j \equiv i \text{ (mod $N(\lambda)$)}} X_j
\end{displaymath}
and so Lemma~\ref{lem:std-7} shows that $\tau$ induces an isomorphism
\begin{displaymath}
\Delta_{\lambda}^{G(b)} \cap \Delta_{\lambda}[\gamma_i(\lambda)] \to \Delta_{\lambda}^{G(b)} \cap \Delta_{\lambda}[\gamma_{i+1}(\lambda)]
\end{displaymath}
for $i \in \bZ/N(\lambda)$. Since the action of $\Aut(\lambda)$ on $\Delta_{\lambda}$ commutes with $G$, and in particular $\tau$, it follows that we get a similar isomorphism after passing to an $\Aut(\lambda)$ isotypic component, and so $\tau(\ol{X}_i)=\ol{X}_{i+1}$. Since $\tau^{N(\lambda)}$ and $\sigma^{N(\lambda)}$ act in the same way on $\Delta_{\lambda}^{G(b)}$, they act the same on $\Delta_{\lambda,\zeta}^{G(b)}$ as well; $\sigma^{N(\lambda)}$ acts by $\zeta$ on this space by definition.
\end{proof}

We can now complete the proof of the theorem.

\begin{lemma} \label{lem:std-9}
If $\Delta_{\lambda,\zeta}$ is not simple then $\ol{A}_2(\lambda)$ is its unique proper non-zero submodule.
\end{lemma}

\begin{proof}
Suppose $M$ is a proper non-zero $\ul{G}$-submodule of $\Delta_{\lambda,\zeta}$. We show $M=\ol{A}_2(\lambda)$. By Lemma~\ref{lem:std-4}, we must have $M \subset \ol{A}_2(\lambda)$, otherwise $M$ would be all of $\Delta_{\lambda,\zeta}$. We now prove the reverse containment. Since $M$ is non-zero, it contains some $L_{\gamma_i(\lambda)}$, and in particular, contains its $G(b)$-invariant space $\ol{X}_i$. Lemma~\ref{lem:std-8} implies that $M$ contains $\ol{X}_j$ for each $j$. Since $L_{\gamma_j(\lambda)}$ is $\ul{H}$-simple and has non-zero intersection with $M$, it follows that $M$ contains $L_{\gamma_j(\lambda)}$. Thus $M$ contains $\ol{A}_2(\lambda)$, as required.
\end{proof}

\section{Standard filtration of induced modules}

\subsection{Statement of results}

Recall that for a weight $\lambda$ we put $\lambda^+=\lambda \wa$ and $\lambda^-=\lambda \wb$, and $I_{\lambda}$ is the induced module $\Ind_H^G(L_{\lambda})$.

\begin{theorem} \label{thm:delta-seq}
For a non-empty weight $\lambda$ we have a natural short exact sequence
\begin{displaymath}
0 \to \Delta_{\lambda^+} \oplus \Delta_{\lambda^-} \to I_{\lambda} \to \Delta_{\lambda} \to 0.
\end{displaymath}
\end{theorem}

Before giving the proof, we note a few corollaries.

\begin{corollary} \label{cor:ind-factor}
Let $\lambda$ be a weight of length $n$ and let $X$ be a $\ul{G}$-module such that every simple $\ul{H}$-module in it has length $\le n$. Then any map $I_{\lambda} \to X$ factors through $\Delta_{\lambda}$.
\end{corollary}

\begin{proof}
The modules $\Delta_{\lambda^{\pm}} \subset I_{\lambda}$ are generated as $\ul{G}$-modules by $\ul{H}$-modules of length $n+1$. These generating $\ul{H}$-submodules must map to~0 in $X$, and so all of $\Delta_{\lambda^{\pm}}$ must map to zero. The result therefore follows from the theorem.
\end{proof}

\begin{corollary}
Every non-trivial simple $\ul{G}$-module is isomorphic to some $M_{\lambda,\zeta}$.
\end{corollary}

\begin{proof}
Let $M$ be a non-trivial simple $G$-module, and let $\lambda$ be a weight of maximum length such that $L_{\lambda}$ appears in $M$; put $n=\ell(\lambda)$. By Frobenius reciprocity, we have a non-zero map $I_{\lambda} \to M$, which is surjective since $M$ is simple. By Corollary~\ref{cor:ind-factor}, this map induces a surjection $\Delta_{\lambda} \to M$. Since $\Delta_{\lambda}$ is the direct sum of $\Delta_{\lambda,\zeta}$'s, it follows that $M$ is a quotient of some $\Delta_{\lambda,\zeta}$. Since $M_{\lambda,\zeta}$ is the unique simple quotient of $\Delta_{\lambda,\zeta}$, it follows that $M$ is isomorphic to $M_{\lambda,\zeta}$.
\end{proof}

\begin{remark}
If $\lambda=\emptyset$ and we interpret $\Delta_{\lambda}$ to be the trivial representation, then we obtain a sequence like the one in Theorem~\ref{thm:delta-seq}, except that we lose exactness on the left: the intersection of $\Delta_{\wa}$ and $\Delta_{\wb}$ in $I_{\emptyset}$ is a copy of the trivial representation. See Figure~\ref{fig:proj}.
\end{remark}

\subsection{Proof}

We fix a weight $\lambda$ of length $n>0$ in what follows. By Theorem~\ref{thm:std}, $\Delta_{\lambda}$ contains a unique copy of $L_{\lambda}$, which generates it as a $\ul{G}$-module. It follows that there is a unique (up to scaling) non-zero map $I_{\lambda} \to \Delta_{\lambda}$, and it is surjective. Let $K$ be the kernel. We must show that $K$ is naturally isomorphic to $\Delta_{\lambda^+} \oplus \Delta_{\lambda^-}$.

For $\phi \in \cC(\bS^{\{n+1\}})$, let $\phi_{\infty} \in \cC(\bR^{(n)})$ be the function given by
\begin{displaymath}
\phi_{\infty}(x_1, \ldots, x_n) = \phi(x_1, \ldots, x_n, \infty).
\end{displaymath}

\begin{lemma} \label{lem:delta-seq-0}
Identify $\Ind(\cC(\bR^{(n)}))$ with $\cC(\bS^{\{n+1\}})$. For an $\ul{H}$-submodule $M$ of $\cC(\bR^{(n)})$, the $\ul{G}$-module $\Ind(M)$ is identified with the subspace of $\cC(\bS^{\{n+1\}})$ consisting of functions $\phi$ such that $(g\phi)_{\infty}$ belongs to $M$ for all $g \in G$. In particular, if $N$ is a $\ul{G}$-submodule of $\cC(\bS^{\{n+1\}})$ then $N \subset \Ind(M)$ if and only if $\phi_{\infty} \in M$ for all $\phi \in N$.
\end{lemma}

\begin{proof}
The induction $\Ind(\cC(\bR^{(n)}))$ consists of functions $\phi \colon G \to \cC(\bR^{(n)})$ that are left $H$-equivariant and right $G$-smooth. One easily sees that such functions correspond to $G$-smooth functions $G \times^H \bR^{(n)} \to k$, where the domain here is the quotient of the usual product by the relations $(gh,x)=(g,hx)$ for $h \in H$. We have an isomorphism
\begin{displaymath}
G \times^H \bR^{(n)} \to \bS^{\{n+1\}}, \qquad
(g, x_1, \ldots, x_n) \mapsto (gx_1, \ldots, gx_n, g\infty).
\end{displaymath}
In this way, we identify  $\Ind(\cC(\bR^{(n)}))$ with $\cC(\bS^{\{n+1\}})$.

Now, let $\phi \in \cC(\bS^{\{n+1\}})$ be given. The function $\phi' \colon G \times^H \bR^{(n)} \to k$ corresponding to $\phi$ is given by
\begin{displaymath}
\phi'(g, x_1, \ldots, x_n)=\phi(gx_1, \ldots, gx_n, g\infty).
\end{displaymath}
Clearly, $\phi'$ belongs to $\Ind(M)$ if and only if $\phi'(g, -)$ belongs to $M$ for each $g \in G$. As $\phi'(g, -)$ is the function $(g^{-1} \phi)_{\infty}$ the main claim follows.

As for the final sentence the condition is clearly necessary. Suppose that it holds. Let $\phi \in N$ be given. Since $N$ is a $\ul{G}$-submodule, we have $g \phi \in N$ for any $g \in G$. Thus $(g \phi)_{\infty} \in M$ for all $g \in G$ by hypothesis. Hence $\phi \in \Ind(M)$, which proves the claim.
\end{proof}

\begin{lemma} \label{lem:delta-seq-1}
Identify $\Ind(\cC(\bR^{(n)}))$ with $\cC(\bS^{\{n+1\}})$. Then $I_{\lambda}$ is identified with a $G$-submodule of $\cC(\bS^{\{n+1\}})$ that contains $\Delta_{\lambda^+}$ and $\Delta_{\lambda^-}$.
\end{lemma}

\begin{proof}
We apply the final sentence of Lemma~\ref{lem:delta-seq-0} with $M=L_{\lambda}$ and $N=\Delta_{\lambda^+}$. It suffices to show that $\phi_{\infty} \in M$ for the generators $\phi$ of $N$. Thus suppose $\bI$ is a cyclicly ordered tuple of intervals of type $\lambda$, and put $\phi=\psi_{\bI}$. If $\infty \not\in I_{n+1}$ then $\phi_{\infty}=0$, which belongs to $M$. Suppose now that $\infty \in I_{n+1}$. Then $\bJ=(I_1,\ldots,I_n)$ is a totally ordered tuple of intervals in $\bR$ of type $\lambda$, and $\phi_{\infty}=\phi_{\bJ}$, which belongs to $M$. Thus $N \subset \Ind(M)$. Then $\lambda^-$ case is similar.
\end{proof}

The following is the key lemma in the proof of the theorem.

\begin{lemma} \label{lem:delta-seq-2}
$\Delta_{\lambda^+}$ and $\Delta_{\lambda^-}$ are linearly disjoint subspaces of $\cC(\bS^{\{n+1\}})$.
\end{lemma}

\begin{proof}
Let $X=\Delta_{\lambda^+} + \Delta_{\lambda^-}$, where the sum is taken inside of $\cC(\bS^{\{n+1\}})$. Let $b \in \bS^{\{n+1\}}$ with $b_{n+1}=\infty$. We know that $\Delta_{\lambda^{\pm}}$ contains $n+1$ simple $H$-modules of length $n$, and $n+1$ of length $n+1$ (Theorem~\ref{thm:std}). The length $n+1$ modules in $\Delta_{\lambda^+}$ and $\Delta_{\lambda^-}$ are non-isomorphic, and thus intersect to~0. It thus suffices to show $X$ contains $2n+2$ simple $H$-modules of length $n$. For this, it is enough to show that $X^{G(b)}$ has dimension at least $2n+2$.

Our usual trick for constructing co-invariants is to use evaluation maps; see, e.g., \cite[\S 4.2]{line}. Here there are not enough evaluation maps, so we use a variant of this idea: we evaluate at all but one coordinate and integrate the remaining coordinate over an interval. By choosing the intervals correctly, we can ensure that the integral is zero on one summand of $X$ and non-zero on the other, and this will let us show that these functionals are linearly independent.

Define functionals $\alpha_0, \beta_0 \in \cC(\bS^{\{n+1\}})^*$ by
\begin{align*}
\alpha_0(\phi) &= \int_{(b_n,b_{n+1})} \phi(b_1, \ldots, b_n, t) dt \\
\beta_0(\phi) &= \int_{(b_{n+1}, b_1)} \phi(b_1, \ldots, b_n, t) dt.
\end{align*}
Let $\alpha_i$ and $\beta_i$ be the $i$th cyclic shifts of $\alpha_0$ and $\beta_0$, for $i \in \bZ/(n+1)$. It is clear that $\alpha_0$ and $\beta_0$ are $G(b)$-equivariant. We claim that these $2n+2$ functionals are linearly independent on $X$, which will complete the proof.

Let $\bI=(I_1, \ldots, I_{n+1})$ be a cyclicly ordered tuple of intervals of type $\lambda^+$ such that $b_i$ belongs to the interior of $I_i$ for all $i$. We have
\begin{displaymath}
\alpha_0(\phi_{\bI})=\vol((b_n, b_{n+1}) \cap I_{n+1}) = -1, \qquad
\beta_0(\phi_{\bI})=\vol((b_{n+1}, b_1) \cap I_{n+1}) = 0.
\end{displaymath}
Indeed, in the first case the intersection is an open interval, and in the second it is a half-open interval. We also have $\alpha_i(\phi_{\bI})=\beta_i(\phi_{\bI})=0$ for all $i \neq 0$. Taking $\bI$ of type $\lambda^-$ leads to similar results, but with the roles of $\alpha$ and $\beta$ reversed. This completes the proof.
\end{proof}

\begin{proof}[Proof of Theorem~\ref{thm:delta-seq}]
In what follows, we write $\ell(M)$ for the length (i.e., number of simple constituents) of $M$ as an $\ul{H}$-module. We have
\begin{displaymath}
\ell(\Delta_{\lambda})=2n, \qquad \ell(\Delta_{\lambda^{\pm}}) = 2n+2, \qquad \ell(I_{\lambda}) = 6n+4.
\end{displaymath}
The first two formula follow from Theorem~\ref{thm:std}, while the third follows from Proposition~\ref{prop:ind}. We thus see that $\ell(K)=4n+4$. From Lemma~\ref{lem:delta-seq-1}, we see that $\Delta_{\lambda^{\pm}}$ is naturally contained in $I_{\lambda}$. Since these modules are generated by $\ul{H}$-submodules of length $n+1$, the generators must map to~0 in $\Delta_{\lambda}$. We thus find that $\Delta_{\lambda^{\pm}} \subset K$. From Lemma~\ref{lem:delta-seq-2}, we see that $\Delta_{\lambda^+} \oplus \Delta_{\lambda^-} \subset K$. Since the two sides have the same $\ul{H}$-length, this inclusion is an equality.
\end{proof}

\section{Special and generic modules} \label{s:fine}

\subsection{Statement of results} \label{ss:fine-intro}

For an integer $n$, let $\pi(n)$ be the length $\vert n \vert$ weight consisting of all $\wa$'s if $n \ge 0$, and all $\wb$'s if $n \le 0$. Also, put $\epsilon(n)=(-1)^{n+1}$. We say that a pair $(\lambda, \zeta)$ with $\lambda$ non-empty is \defn{special} if $\lambda=\pi(n)$ and $\zeta=\epsilon(n)$ for some integer $n$; otherwise, we say that $(\lambda, \zeta)$ is \defn{generic}. We also apply this terminology to the simple modules $M_{\lambda,\zeta}$, and label the trivial module $\bbone$ as special.

\begin{theorem} \label{thm:fine}
Let $(\lambda, \zeta)$ be given with $\lambda$ non-empty.
\begin{enumerate}
\item If $(\lambda, \zeta)$ is generic then $\Delta_{\lambda,\zeta}=M_{\lambda,\zeta}$, and this module is both projective and injective.
\item Suppose that $(\lambda, \zeta)=(\pi(n), \epsilon(n))$ is special with $n>0$. Then:
\begin{enumerate}[(i)]
\item The standard module $\Delta_{\pi(n),\epsilon(n)}$ has length two.
\item The unique proper non-zero submodule of $\Delta_{\pi(n),\epsilon(n)}$ is isomorphic to $M_{\pi(n-1),\epsilon(n-1)}$ (which is taken to mean $\bbone$ if $n=1$).
\item Let $b \in \bS^{\{n+1\}}$. Then $M_{\lambda,\zeta}^{G(b)}$ is one-dimensional, and the standard generator $\tau$ of $G[b]/G(b)$ acts on this space by $-\epsilon(n)$.
\end{enumerate}
\end{enumerate}
\end{theorem}

We note that if $(\lambda,\zeta)$ is generic with $\ell(\lambda)=n$ then $M_{\lambda,\zeta}=\Delta_{\lambda,\zeta}$ contains $\ul{H}$-simples of lengths $n$ and $n-1$. On the other hand, if $(\lambda,\zeta)$ is special then $M_{\lambda,\zeta}$ is irreducible as an $\ul{H}$-module, as it is simply isomorphic to $L_{\lambda}$ by Theorem~\ref{thm:std}.

\begin{corollary}
The simples $M_{\lambda,\zeta}$ and $M_{\mu,\omega}$ are isomorphic if and only if $\omega=\zeta$ and $\mu$ is a cyclic shift of $\lambda$ (i.e., $\mu=\sigma^i(\lambda)$ for some $i$).
\end{corollary}

\begin{proof}
If $\omega=\zeta$ and $\mu$ is a cyclic shift of $\lambda$ then we have seen that $\Delta_{\lambda,\zeta}$ is isomorphic to $\Delta_{\mu,\omega}$, and so it follows that $M_{\lambda,\zeta}$ is isomorphic to $M_{\mu,\omega}$. Now suppose that $M_{\lambda,\zeta}$ is isomorphic to $M_{\mu,\omega}$. By the preceding remarks, $(\lambda,\zeta)$ and $(\mu,\omega)$ are both generic or both special. If both are generic then the result follows from Corollary~\ref{cor:std-isom}. Now suppose both are special. Then as $\ul{H}$-modules we have $M_{\lambda,\zeta}=L_{\lambda}$ and $M_{\mu,\omega}=L_{\mu}$, and so $\lambda=\mu$. Since $\zeta=(-1)^{\ell(\lambda)+1}$ and $\omega=(-1)^{\ell(\mu)+1}$, we also have $\omega=\zeta$.
\end{proof}

\begin{corollary}
For $n>0$, we have $M_{\pi(n), \epsilon(n)} \cong \lw^n(M_{\wa})$.
\end{corollary}

\begin{proof}
The $\ul{H}$-module underlying $M_{\wa}$ is $L_{\wa}$. We have $\lw^n(L_{\wa})=L_{\pi(n)}$ by \cite[Proposition~8.5]{line}. Thus $\lw^n(M_{\wa})$ is a $\ul{G}$-module whose underlying $\ul{H}$-module is $L_{\pi(n)}$. It is therefore irreducible as a $\ul{G}$-module, and must be isomorphic to $M_{\pi(n),\epsilon(n)}$ by the classification of irreducibles and Theorem~\ref{thm:fine}.
\end{proof}

Theorem~\ref{thm:fine} also implies the decomposition of $\cC=\uRep^{\rf}(G)$ stated in \S \ref{ss:results}(b). We now explain. Recall that $\cC_{\rm gen}$ (resp.\ $\cC_{\rm sp}$) is the full subcategory of $\cC$ spanned by modules whose simple constituents are all generic (resp.\ special). These categories are clearly orthogonal in the sense that $\Hom_{\cC}(M,N)=\Hom_{\cC}(N,M)=0$ for $M \in \cC_{\rm gen}$ and $N \in \cC_{\rm sp}$. Let $M$ be a given object of $\cC$. Since every generic simple is projective and injective, any such constituent of $M$ splits off as a summand. We thus obtain a canonical decomposition $M=M_{\rm gen} \oplus M_{\rm sp}$ with $M_{\rm gen} \in \cC_{\rm gen}$ and $M_{\rm sp} \in \cC_{\rm sp}$; to see that this is canonical, simply note that $M_{\rm gen}$ is the maximal subobject of $M$ belonging to $\cC_{\rm gen}$, and similarly for $M_{\rm sp}$. We thus find that $\cC=\cC_{\rm gen} \oplus \cC_{\rm sp}$, and that $\cC_{\rm gen}$ is semi-simple. The precise structure of $\cC_{\rm sp}$ will be determined in \S \ref{s:special} below.

\subsection{The special case}\label{ss:fine-special}

For $1 \le i \le n+1$, let $p_i \colon \bS^{\{n+1\}} \to \bS^{\{n\}}$ be the projection away from the $i$th coordinate. Define $d_n \colon \cC(\bS^{\{n\}}) \to \cC(\bS^{\{n+1\}})$ to be $\sum_{i=1}^{n+1} (-1)^{i+1} p_i^*$.

\begin{proposition} \label{prop:fine-aux}
The map $d_n$ carries $\Delta_{\pi(n),\epsilon(n)}$ into $\Delta_{\pi(n+1),\epsilon(n+1)}$. The sequence of maps
\begin{displaymath}
\xymatrix@C=3em{
0 \ar[r] & \bbone \ar[r]^-{d_0} & \Delta_{\pi(1),\epsilon(1)} \ar[r]^-{d_1} & \Delta_{\pi(2),\epsilon(2)} \ar[r]^-{d_2} & \cdots }
\end{displaymath}
is an exact complex.
\end{proposition}

Before proving the proposition, we show how it implies the special case of the theorem.

\begin{proof}[Proof of Theorem~\ref{thm:fine}(b)]
No map in the complex is surjective (since $\Delta_{\pi(n),\epsilon(n)}$ contains simple $\ul{H}$-modules not appearing in $\Delta_{\pi(n-1),\epsilon(n-1)}$ by Theorem~\ref{thm:std}), and so no map in the complex is zero (besides the leftmost one). Thus the terms of the complex are not simple, and thus of length two by Theorem~\ref{thm:std}. It follows that the image of $d_{n-1}$ must be the unique proper non-zero submodule of $\Delta_{\pi(n),\epsilon(n)}$; since this is a simple quotient of $\Delta_{\pi(n-1),\epsilon(n-1)}$, it must be $M_{\pi(n-1),\epsilon(n-1)}$.

Now, let $b \in \bS^{\{n+1\}}$ be given. We have an exact sequence
\begin{displaymath}
0 \to M_{\pi(n),\epsilon(n)} \to \Delta_{\pi(n+1),\epsilon(n+1)} \to M_{\pi(n+1),\epsilon(n+1)} \to 0.
\end{displaymath}
Since $\uRep(G(b))$ is semi-simple, formation of $G(b)$-invariants is exact. As the rightmost module has no invariants, we obtain an isomorphism $M_{\pi(n),\epsilon(n)}^{G(b)} \cong \Delta^{G(b)}_{\pi(n+1),\epsilon(n+1)}$. We have already seen in Theorem~\ref{thm:std} that the standard generator of $G[b]/G(b)$ acts on this space by $\epsilon(n+1)$.

This completes the proof of this case of the theorem when $\lambda$ consists of all $\wa$'s. Of course, the case when $\lambda$ consists of all $\wb$'s is exactly the same.
\end{proof}

Let $\bI=(I_1,\ldots,I_n)$ be an ordered tuple of intervals in $\bS$ of type $\pi(n)$. By Proposition~\ref{prop:merely-ordered} $\psi_{\bI}$ belongs to $\Delta_{\pi(n)}$. Let $\theta_{\bI}$ be $n$ times the projection of $\psi_{\bI}$ to $\Delta_{\pi(n),\epsilon(n)}$. That is, if $n$ is odd then $\theta_{\bI}=\sum_{i \in \bZ/n} (\sigma^i)^*(\psi_{\bI})$, while if $n$ is even then $\theta_{\bI}=\sum_{i \in \bZ/n} (-1)^i (\sigma^i)^*(\psi_{\bI})$. The following lemma is the key computation needed to prove the proposition.

\begin{lemma} \label{lem:fine-1}
Let $\bS = I_1 \sqcup \cdots \sqcup I_{n+1}$ be a decomposition where each $I_i$ is a half-open interval of type $\wa$. Put $\bJ=(I_1, \ldots, I_n)$ and $\bI=(I_1, \ldots, I_{n+1})$. Then $d_n(\theta_{\bJ})=(-1)^n \cdot \theta_{\bI}$.
\end{lemma}

\begin{proof}
Put $\alpha=d_n(\theta_{\bJ})$. It suffices to establish the following three statements:
\begin{enumerate}
\item $\alpha(x_1, \ldots, x_{n+1})=(-1)^n$ if $x_i \in I_i$ for all $1 \le i \le n+1$.
\item $\alpha(x_1, \ldots, x_{n+1})=0$ if two $x_i$'s belong to the same $I_j$.
\item $\sigma^*(\alpha)=\epsilon(n+1) \alpha$.
\end{enumerate}
Indeed, this implies that $\alpha=(-1)^n \cdot \theta_{\bI}$.

(a) Suppose $x_i \in I_i$ for all $1 \le i \le n+1$. By definition,
\begin{equation} \label{eq:alpha}
\alpha(x_1, \ldots, x_{n+1})=\sum_{i=1}^{n+1} (-1)^{i+1} \theta_{\bJ}(x_1, \ldots, \hat{x}_i, \ldots, x_{n+1})
\end{equation}
In this sum, only the $i=n+1$ term is non-zero, and it is equal to $(-1)^n$. We thus find $\alpha(x_1, \ldots, x_{n+1})=(-1)^{n+1}$, as required.

(b) Suppose $x \in \bS^{\{n+1\}}$ is given and two coordinates belongs to the same $I_i$. Since the $x_i$'s are ordered, there are indices $i$ and $j$ such that $x_i$ and $x_{i+1}$ both belong to $I_j$. If $j=n+1$ then every term in the sum in \eqref{eq:alpha} vanishes, as some coordinate belongs to $I_{n+1}$. Suppose $j \ne n+1$. All terms but the $i$ and $i+1$ ones in the sum in \eqref{eq:alpha} vanish, as there are two coordinates in the same interval in $\bJ$. Since $x_i$ and $x_{i+1}$ belong to the same interval in $\bJ$, we have
\begin{displaymath}
\theta_{\bJ}(x_1, \ldots, \hat{x}_i, \ldots, x_{n+1}) =
\theta_{\bJ}(x_1, \ldots, \hat{x}_{i+1}, \ldots x_{n+1}).
\end{displaymath}
Thus the $i$ and $i+1$ terms in the sum in \eqref{eq:alpha} cancel. Hence $\alpha(x)=0$, as required.

(c) We have
\begin{align*}
& (\sigma^* \alpha)(x_1, \ldots, x_{n+1})
= \alpha(x_{n+1}, x_1, \ldots, x_n) \\
&= \theta_{\bJ}(x_1, \ldots, x_n) + \sum_{i=1}^n (-1)^i \theta_{\bJ}(x_{n+1}, x_1, \ldots, \hat{x}_i, \ldots, x_n) \\
&= \theta_{\bJ}(x_1, \ldots, x_n) + \sum_{i=1}^n (-1)^{i+n+1} \theta_{\bJ}(x_1, \ldots, \hat{x}_i, \ldots, x_n, x_{n+1}) \\
&= \sum_{i=1}^{n+1} (-1)^{i+n+1} \theta_{\bJ}(x_1, \ldots, \hat{x}_i, \ldots, x_{n+1})
= (-1)^n \alpha(x_1, \dots, x_{n+1}).
\end{align*}
In the third step, we used that $\theta_{\bJ}$ is an eigenvector for $\sigma^*$ of eigenvalue $\epsilon(n)=(-1)^{n+1}$. We thus see that $\alpha$ is an eigenvector for $\sigma^*$ with eigenvalue $\epsilon(n+1)$.
\end{proof}

\begin{lemma} \label{lem:fine-2}
With the same notation as Lemma~\ref{lem:fine-1}, $\theta_{\bJ}$ generates $\Delta_{\pi(n),\epsilon(n)}$ as a $\ul{G}$-module.
\end{lemma}

\begin{proof}
Without loss of generality, we suppose $\infty$ belongs to $I_{n+1}$. Let $i \colon \bR^n \to \bS^{\{n\}}$ denote the standard inclusion. Then $i^*(\theta_{\bJ})=\phi_{\bJ}$ generates $L_{\lambda}$ as an $\ul{H}$-module. It follows that the $\ul{G}$-submodule of $\Delta_{\pi(n),\epsilon(n)}$ generated by $\theta_{\bJ}$ contains $L_{\lambda}$, and is thus all of $\Delta_{\pi(n),\epsilon(n)}$ by Theorem~\ref{thm:std}.
\end{proof}

\begin{proof}[Proof of Proposition~\ref{prop:fine-aux}]
Let $\bI$ and $\bJ$ be as in Lemma~\ref{lem:fine-1}. By that lemma, $d_n(\theta_{\bJ})=(-1)^n \cdot \theta_{\bI}$ belongs to $\Delta_{\pi(n+1),\epsilon(n+1)}$. By Lemma~\ref{lem:fine-2}, it follows that $d_n$ maps all of $\Delta_{\pi(n),\epsilon(n)}$ into $\Delta_{\pi(n+1),\epsilon(n+1)}$. Note that $d_n(\Delta_{\pi(n),\epsilon(n)})$ is a non-zero by the computation of $d_n(\theta_{\bJ})$.

The proposition now essentially follows from Theorem~\ref{thm:std}. Since $d_n$ cannot be surjective (by comparing the underlying $\ul{H}$-modules), each $\Delta_{\pi(n),\epsilon(n)}$ must have length two, and the image of $d_n$ must be the unique proper non-zero submodule of $\Delta_{\pi(n+1),\epsilon(n+1)}$. Thus the complex is exact.
\end{proof}

\subsection{The generic case}

We now prove Theorem~\ref{thm:fine}(a), in two steps.

\begin{lemma}
If $(\lambda,\zeta)$ is generic then $\Delta_{\lambda,\zeta}$ is irreducible.
\end{lemma}

\begin{proof}
Consider the following statement:
\begin{itemize}
\item[$(S_n)$] If $\Delta_{\lambda,\zeta}$ is reducible, with $\ell(\lambda)=n$, then $(\lambda,\zeta)$ is special.
\end{itemize}
We prove the statement $(S_n)$ by induction on $n$. The statement is vacuously true for $n \le 1$ since then any $(\lambda,\zeta)$ is special.

Suppose now that $(S_{n-1})$ holds, and let $\Delta_{\lambda,\zeta}$ be reducible with $\ell(\lambda)=n$. We show that $(\lambda,\zeta)$ is special. By Theorem~\ref{thm:std}, $\Delta_{\lambda,\zeta}$ has a unique non-zero proper submodule $M$, which is irreducible and isomorphic to $\ol{A}_2(\lambda)$ as an $\ul{H}$-module. Now, $M$ must be isomorphic to $M_{\mu,\omega}$ for some $(\mu,\omega)$ where $\ell(\mu)=n-1$. If $(\mu,\omega)$ is not special then $M_{\mu,\omega}=\Delta_{\mu,\omega}$ has $\ul{H}$-simples of length $n-1$ and $n-2$, and thus cannot be $M$. Thus $(\mu,\omega)$ must be special. Without loss of generality, say $(\mu,\omega)=(\pi(n-1),\epsilon(n-1))$. Since $\ol{A}_2(\lambda)=L_{\pi(n-1)}$, it follows that $\lambda=\pi(n)$.

Now, let $b \in \bS^{\{n\}}$, and let $\tau$ be the standard generator of $G[b]/G(b)$. By Theorem~\ref{thm:std}, $\tau$ acts by $\zeta$ on $\Delta_{\lambda,\zeta}^{G(b)}=M^{G(b)}$. On the other hand, by Theorem~\ref{thm:fine}(b) (which we have already proved in \S \ref{ss:fine-special}), $\tau$ acts by $\epsilon(n)$ on $M_{\pi(n-1),\epsilon(n-1)}^{G(b)}$. Thus $\zeta=\epsilon(n)$, and so $(\lambda,\zeta)$ is special.
\end{proof}

\begin{lemma}
If $(\lambda,\zeta)$ is generic then $\Delta_{\lambda,\zeta}$ is both projective and injective.
\end{lemma}

\begin{proof}
Let $n=\ell(\lambda)$ and let $P$ be the projective cover of $M_{\lambda,\zeta}$. Since there is a surjection $I_{\lambda} \to M_{\lambda,\zeta}$, it follows that $P$ is a summand of $I_{\lambda}$. In particular, $P$ is also injective. (This is true in any pre-Tannakian category; see \cite[Proposition~6.1.3]{EGNO}.)

Suppose that $\mu$ is a weight of length $n+1$. Then
\begin{displaymath}
0=\Hom_{\ul{G}}(P, I_{\mu})=\Hom_{\ul{H}}(P, L_{\mu}).
\end{displaymath}
The first $\Hom$ space counts the multiplicity of $M_{\lambda,\zeta}$ in $I_{\mu}$, and this must vanish since $M_{\lambda,\zeta}=\Delta_{\lambda,\zeta}$ has an $\ul{H}$-simple of length $n-1$, but $I_{\mu}$ does not by Proposition~\ref{prop:ind}. The second equality above is Frobenius reciprocity. We conclude that $P$ is concentrated in degrees $n$ and $n-1$.

By Corollary~\ref{cor:ind-factor}, we find that the quotient map $I_{\lambda} \to P$ factors through $\Delta_{\lambda}$, and so $P$ is a summand of $\Delta_{\lambda}$. Since $\Delta_{\lambda}=\bigoplus \Delta_{\lambda,\omega}$ is the indecomposable decomposition of $\Delta_{\lambda}$, we must have $P=\Delta_{\lambda,\zeta}$.
\end{proof}

\section{The special block} \label{s:special}

\subsection{Initial remarks}

Recall that $\cC_{\rm sp}$ is the full subcategory of $\cC=\uRep^{\rf}(G)$ spanned by objects whose simple constituents are all special. Our goal in this section is to determine the structure of this category.

Let $n$ be an integer (possibly negative). Recall that $\pi(n)$ is the word of length $\vert n \vert$ consisting of all $\wa's$ if $n \ge 0$, and all $\wb$'s if $n \le 0$; also, we put $\epsilon(n)=(-1)^{n+1}$. We now introduce some additional notation:
\begin{itemize}
\item We let $M(n)$ be the simple $M_{\pi(n),\epsilon(n)}$ for $n \ne 0$, and we put $M(0)=\bbone$.
\item For $n>0$, we let $\Delta(n)=\Delta_{\pi(n),\epsilon(n)}$. For $n \le 0$, we let $\Delta(n)=\Delta(1-n)^{\vee}$.
\item For $n \le 0$, we let $\nabla(n)=\Delta_{\pi(n-1),\epsilon(n-1)}$. For $n>0$, we let $\nabla(n)=\nabla(1-n)^{\vee}$.
\end{itemize}
The definitions of $\Delta(n)$ and $\nabla(n)$ may appear somewhat strange, but we will see that they work very well; see Figure~\ref{fig:proj} for some motivation. Here are some small examples of how the indexing works:
\begin{align*}
\Delta(-1) &= \Delta_{\wa\wa,-1}^{\vee} & \Delta(0) &= \Delta_{\wa,+1}^{\vee} & \Delta(1) &= \Delta_{\wa,+1} & \Delta(2) &= \Delta_{\wa\wa,-1} \\
\nabla(-1) &= \Delta_{\wb\wb,-1} & \nabla(0) &= \Delta_{\wb,+1} & \nabla(1) &= \Delta_{\wb,+1}^{\vee} & \nabla(2) &= \Delta^{\vee}_{\wb\wb,-1}
\end{align*}
Note that $\Delta_{\wa,+1}=\Delta_{\wa}$, so the sign could be omitted in this case. The following propositions give some basic information about these objects.  Recall that $(-)^{\dag}$ denotes transpose (see \S \ref{ss:transp}).

\begin{proposition} \label{prop:dual}
For any $n \in \bZ$, we have $M(n)^{\vee} \cong M(-n)$.
\end{proposition}

\begin{proof}
From \cite[Proposition~4.16]{line}, we have $L_{\pi(n)}^{\vee} \cong L_{\pi(-n)}$. Thus $M(n)^{\vee}$ is a simple $\ul{G}$-module with underlying $\ul{H}$-module $L_{\pi(-n)}$, and must therefore be isomorphic to $M(-n)$.
\end{proof}

\begin{proposition}
For any $n \in \bZ$, we have non-split exact sequences
\begin{displaymath}
0 \to M(n-1) \to \Delta(n) \to M(n) \to 0
\end{displaymath}
and
\begin{displaymath}
0 \to M(n) \to \nabla(n) \to M(n-1) \to 0
\end{displaymath}
\end{proposition}

\begin{proof}
The first sequence follows from Theorem~\ref{thm:fine}(b) for $n>0$, and then for $n \le 0$ by duality. The second sequence is similar. These sequences are non-split since $\Delta_{\lambda,\zeta}$ is always indecomposable.
\end{proof}

\begin{proposition} \label{prop:sp-transp}
For any $n \in \bZ$, we have $M(n)^{\dag} \cong M(-n)$ and $\Delta(n)^{\dag} = \nabla(1-n)$.
\end{proposition}

\begin{proof}
Since $M(n)^{\dag}$ has underlying $\ul{H}$-module $L_{\pi(n)}^{\dag} \cong L_{\pi(-n)}$, it is necessarily isomorphic to $M(-n)$. For $n>0$, we have $\Delta(n)^{\dag} \cong \Delta_{\pi(-n),\epsilon(n)}$ by Remark~\ref{rmk:std-transp}, which is $\nabla(1-n)$. Since transpose is a tensor functor, it commutes with duality, and so the result for $n \le 0$ follows as well.
\end{proof}

\begin{remark}
If we put $M^* = (M^{\vee})^{\dag}$ then $M(n)^*=M(n)$ and $\Delta(n)^*=\nabla(n)$. The presence of several closely related $\bZ/2$ actions has a similar feel to shifted Weyl group actions, and makes it impossible to choose an indexing which looks natural for all three involutions. For example, some of the results above would look more natural if we indexed the $M$'s by half-integers instead of integers.
\end{remark}

\subsection{Projectives}

For $n \in \bZ$, we let $P(n)$ be the projective cover of $M(n)$. We now determine the structure of these objects.

\begin{theorem} \label{thm:proj}
Let $n \in \bZ$ be given. Then we have the following.
\begin{enumerate}
\item $P(n)$ is the injective envelope of $M(n)$.
\item We have $P(n)^{\vee} \cong P(-n)$ and $P(n)^{\dag} \cong P(-n)$.
\item We have a short exact sequence
\begin{displaymath}
0 \to \Delta(n+1) \to P(n) \to \Delta(n) \to 0.
\end{displaymath}
\item We have a short exact sequence
\begin{displaymath}
0 \to \nabla(n) \to P(n) \to \nabla(n+1) \to 0.
\end{displaymath}
\end{enumerate}
\end{theorem}

Parts~(c) and (d) of the theorem are depicted for $n=0,1$ in Figure~\ref{fig:proj}. We prove two special cases of the theorem before giving the general proof.

\begin{figure}
\begin{tikzpicture}
\node (a) at (0,3) {$\bbone$};
\node (b) at (-1.5,1.5) {$M_{\wb,1}$};
\node (c) at (1.5,1.5) {$M_{\wa,1}$};
\node (d) at (0,0) {$\bbone$};
\draw (a) -- (b);
\draw (a) -- (c);
\draw (d) -- (b);
\draw (d) -- (c);
\draw[decorate,decoration={brace,amplitude=10pt,mirror},xshift=12pt,yshift=-12pt] (0,0) -- (1.5,1.5) node [black,midway,xshift=15pt,yshift=-14pt] {$\Delta_{\wa}$};
\draw[decorate,decoration={brace,amplitude=10pt},xshift=-12pt,yshift=-12pt] (0,0) -- (-1.5,1.5) node [black,midway,xshift=-15pt,yshift=-14pt] {$\Delta_{\wb}$};
\draw[decorate,decoration={brace,amplitude=10pt},xshift=-12pt,yshift=12pt] (-1.5,1.5) -- (0,3) node [black,midway,xshift=-15pt,yshift=14pt] {$\Delta_{\wa}^{\vee}$};
\draw[decorate,decoration={brace,amplitude=10pt,mirror},xshift=12pt,yshift=12pt] (1.5,1.5) -- (0,3) node [black,midway,xshift=15pt,yshift=14pt] {$\Delta_{\wb}^{\vee}$};
\end{tikzpicture}
\qquad\qquad\qquad
\begin{tikzpicture}
\node (a) at (0,3) {$M_{\wa,1}$};
\node (b) at (-1.5,1.5) {$\bbone$};
\node (c) at (1.5,1.5) {$M_{\wa \wa,-1}$};
\node (d) at (0,0) {$M_{\wa,1}$};
\draw (a) -- (b);
\draw (a) -- (c);
\draw (d) -- (b);
\draw (d) -- (c);
\draw[decorate,decoration={brace,amplitude=10pt,mirror},xshift=12pt,yshift=-12pt] (0,0) -- (1.5,1.5) node [black,midway,xshift=22pt,yshift=-15pt] {$\Delta_{\wa\wa,-1}$};
\draw[decorate,decoration={brace,amplitude=10pt},xshift=-12pt,yshift=-12pt] (0,0) -- (-1.5,1.5) node [black,midway,xshift=-15pt,yshift=-15pt] {$\Delta_{\wb}^{\vee}$};
\draw[decorate,decoration={brace,amplitude=10pt},xshift=-12pt,yshift=12pt] (-1.5,1.5) -- (0,3) node [black,midway,xshift=-15pt,yshift=15pt] {$\Delta_{\wa}$};
\draw[decorate,decoration={brace,amplitude=10pt,mirror},xshift=12pt,yshift=12pt] (1.5,1.5) -- (0,3) node [black,midway,xshift=22pt,yshift=15pt] {$\Delta_{\wb\wb,-1}^{\vee}$};
\end{tikzpicture}
\caption{Diagrams of $P(0)$ and $P(1)$. Submodules are at the bottom and quotients at the top. The right diagram shows that $P(1)$ contains $\Delta_{\wa\wa,-1}$ as a sub, with quotient $\Delta_{\wa}$. In general, $P(n)$, for $n \ge 1$, contains $\Delta_{\pi(n),\epsilon(n)}$ as a sub with quotient $\Delta_{\pi(n-1),\epsilon(n-1)}$, and a similar patterns holds for $n \le -1$. The left picture above shows that $P(0)$ behaves differently: $\Delta_{\wa}$ is a sub, with quotient $\Delta_{\wa}^{\vee}$. This picture motivates our definition of $\Delta(n)$, which handles these cases uniformly. }
\label{fig:proj}
\end{figure}
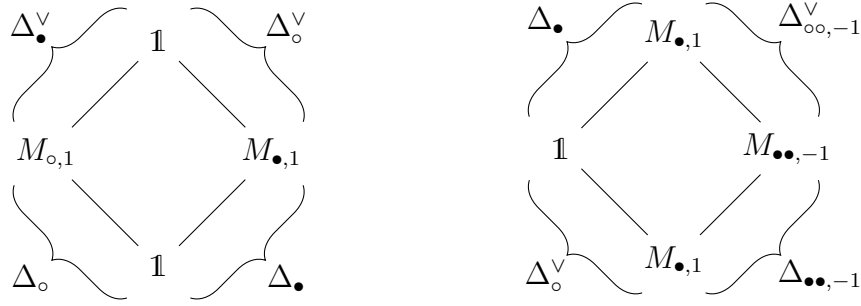

\begin{lemma} \label{lem:proj-1}
Theorem~\ref{thm:proj}(c) holds for $n>0$.
\end{lemma}

\begin{proof}
From Theorem~\ref{thm:delta-seq}, we have a short exact sequence
\begin{displaymath}
0 \to \Delta_{\pi(n+1)} \oplus \Delta_{\pi(n)^-} \to I_{\pi(n)} \to \Delta_{\pi(n)} \to 0,
\end{displaymath}
where $\pi(n)^-$ is the word $\pi(n)$ with $\wb$ appended to the end. Now, $\Delta_{\pi(n)^-}$ is a generic simple, and thus splits off $I_{\lambda}$ as a summand. The module $\Delta_{\pi(n)}$ is the sum of the modules $\Delta_{\pi(n),\zeta}$ as $\zeta$ varies over all $n$th roots of unity. Except for $\zeta=\epsilon(n)$, these are generic simples, and thus split off. Similarly for the $\Delta_{\pi(n+1)}$ term. We thus find that $I_{\pi(n)}$ has a summand $I$ that fits into a short exact sequence
\begin{displaymath}
0 \to \Delta(n+1) \to I \to \Delta(n) \to 0.
\end{displaymath}
Since $I$ surjects onto $M(n)$, we see that $P(n)$ is a summand of $I$.

Let $e$ be the idempotent of $I$ such that $P=e(I)$. The $\ul{H}$-simple $L_{\pi(n+1)}$ occurs in $\Delta(n+1)$ with multiplicity one, and does not occur in $\Delta(n)$. It follows that $e$ must map this $L_{\pi(n+1)}$ into $\Delta(n+1)$. Since this $L_{\pi(n+1)}$ generates $\Delta(n+1)$ by Theorem~\ref{thm:std}, we see that $e$ maps $\Delta(n+1)$ into itself. Since $\Delta(n+1)$ and $\Delta(n)$ are indecomposable, it follows that $P$ must be either $\Delta(n)$ or all of $I$.

Suppose now that $P=\Delta(n)$. The kernel of the map $P \to M(n)$ is then $M(n-1)$, and so we have a presentation
\begin{displaymath}
P(n-1) \to P(n) \to M(n) \to 0.
\end{displaymath}
We thus see that $\Ext^1(M(n), M(n+1))$ is a subquotient of $\Hom(P(n-1), M(n+1))=0$. However, we know that this $\Ext^1$ group does not vanish, due to the extension afforded by $\nabla(n+1)$, and so we have a contradiction. Thus $P=I$, which completes the proof.
\end{proof}

\begin{lemma} \label{lem:proj-2}
Theorem~\ref{thm:proj}(c) holds for $n=0$.
\end{lemma}

\begin{proof}
Since $\cC(\bS) \cong I_{\emptyset}$, Proposition~\ref{prop:ind}, gives an $\ul{H}$-decomposition
\begin{displaymath}
\cC(\bS) = L^{\oplus 2}_{\emptyset} \oplus L_{\wa} \oplus L_{\wb}.
\end{displaymath}
We have seen in Example~\ref{ex:std} that there is a unique trivial $\ul{G}$-submodule of $\cC(\bS)$. This is in fact the socle of $\cC(\bS)$: indeed, any other simple $\ul{G}$-submodule would be generated by $L_{\wa}$ or $L_{\wb}$, but these generate $\Delta_{\wa}$ and $\Delta_{\wb}$ (Theorem~\ref{thm:std}), which are not simple. It follows that $\cC(\bS)$ is indecomposable. Since it is projective (being an induction) and surjects onto the trivial representation, it is thus $P(0)$.

Now, $\cC(\bS)$ contains $\Delta_{\wa}=\Delta(1)$ as a submodule. Since $\cC(\bS)$ is self-dual, it thus admits $\Delta(1)^{\vee}=\Delta(0)$ as a quotient. We claim that the sequence
\begin{displaymath}
0 \to \Delta(1) \to P(0) \to \Delta(0) \to 0
\end{displaymath}
is exact. Indeed, $L_{\wa}$ does not occur in $\Delta(0)$, and thus belongs to the kernel of $P(0) \to \Delta(0)$. Since $L_{\wa}$ generates $\Delta_{\wa}=\Delta(1)$, it follows that the above composition is zero. It is therefore exact by length reasons.
\end{proof}

\begin{proof}[Proof of Theorem~\ref{thm:proj}]
Suppose $n \ge 0$. Then (c) holds by Lemmas~\ref{lem:proj-1} and~\ref{lem:proj-2}. Since $P(n)$ is a summand of $I_{\pi(n)}$, which is projective and injective, it follows that $P(n)$ is injective. By (c) we see that $P(n)$ contains $M(n)$ as a submodule. Since $P(n)$ is indecomposable injective, it must therefore be the injective envelope of $M(n)$. It thus follows that $P(n)^{\vee}$ is the projective envelope of $M(n)^{\vee} \cong M(-n)$, and so $P(n)^{\vee} \cong P(-n)$. Since $(-)^{\dag}$ is a covariant equivalence of abelian categories, it follows that $P(n)^{\dag}$ is the projective cover of $M(n)^{\dag} \cong M(-n)$, and so $P(n)^{\dag} \cong P(-n)$. This proves (a,b,c) for $n \ge 0$, and the remaining cases follow from duality. (d) now follows from (c) by taking transpose.
\end{proof}

\subsection{Maps of projectives}\label{s:maps-of-projectives}

The spaces
\begin{displaymath}
\Hom(P(n), P(n-1)), \qquad \Hom(P(n), P(n+1))
\end{displaymath}
are each one dimensional, since $M(n)$ has multiplicity one in $P(n-1)$ and $P(n+1)$. Let $\alpha_n$ and $\beta_n$ be bases of these spaces.

\begin{proposition} \label{prop:P-homs}
We have the following:
\begin{enumerate}
\item $\alpha_{n-1} \alpha_n=0$.
\item $\beta_{n+1} \beta_n=0$.
\item $\beta_{n-1} \alpha_n = c(n) \cdot \alpha_{n+1} \beta_n$ for some non-zero scalar $c(n)$.
\end{enumerate}
\end{proposition}

\begin{proof}
Since $M(n)$ is not a constituent of $P(n-2)$, we have $\Hom(P(n), P(n-2))=0$, and so (a) follows. A similar argument gives (b). Now, we have a diagram
\begin{displaymath}
\xymatrix@R=1em{
&&& \nabla(n) \ar@{^(->}[rd] \\
P(n) \ar[rr]^{\alpha_n} \ar@{->>}[rd] && P(n-1) \ar@{->>}[ru] \ar[rr]^{\beta_{n-1}} && P(n) \\
& \Delta(n) \ar@{^(->}[ru] }
\end{displaymath}
that commutes up to non-zero scalars. The kernel of $P(n-1) \to \nabla(n)$ is $\nabla(n-1)$, which does not contain $M(n)$ as a constituent. It follows that the $M(n)$ in $\Delta(n)$ has non-zero image in $\nabla(n)$. Thus the composition $\beta_{n-1} \alpha_n$ is non-zero. However, it does contain $M(n)$ in its kernel. A similar analysis shows the same for $\alpha_{n+1} \beta_n$. The space of maps $P(n) \to P(n)$ killing $M(n)$ is one-dimensional, and so the result follows.
\end{proof}

\begin{proposition} \label{prop:alpha-beta}
It is possible to choose $\alpha_n$ and $\beta_n$ such that $c(n)=1$ for all $n$.
\end{proposition}

\begin{proof}
Choose the $\alpha_n$'s arbitrarily, let $\beta'_n$ be an arbitrary non-zero map $P(n) \to P(n+1)$, and let $c'(n)$ be the scalar such that $\beta'_{n-1}\alpha_n = c'(n) \alpha_{n+1} \beta'_n$. Put $\beta_0=\beta'_0$. For $n<0$, define $\beta_n=(c'(n+1) \cdots c'(0))^{-1} \beta'_n$. For $n>0$, define $\beta_n=c(1) \cdots c(n) \beta'_n$. This leads to $c(n)=1$ for all $n$.
\end{proof}

\subsection{Description of the category}

Let $R=k[x,y]/(x^2,y^2)$, and $\bZ$-grade $R$ by letting $x$ and $y$ have degrees $-1$ and $+1$. As a $k$-vector space, $R$ is four dimensional, with basis $1, x, y, xy$. Let $\Mod_R$ denote the category of $\bZ$-graded $R$-algebras. For $n \in \bZ$, we let $R(n)$ be the free $R$-module of rank one with generator of degree $n$. The following theorem, which is the main result of \S \ref{s:special}, provides a complete description of $\cC_{\rm sp}$.

\begin{theorem} \label{thm:spequiv}
We have an equivalence of $k$-linear categories
\begin{displaymath}
\Phi \colon \cC_{\rm sp} \to \Mod_R, \qquad P(n) \mapsto R(n).
\end{displaymath}
\end{theorem}

\begin{proof}
Let $\cP$ be the full subcategory of $\cC_{\rm sp}$ spanned by the $P(n)$'s, and similarly let $\cQ$ be the full subcategory of $\Mod_R$ spanned by the $R(n)$'s. Choose $\alpha_n$'s and $\beta_n$'s as in Proposition~\ref{prop:alpha-beta}. Define $\Phi_0 \colon \cP \to \cQ$ by
\begin{displaymath}
\Phi_0(P(n))=R(n), \qquad
\Phi_0(\alpha_n) = x, \qquad
\Phi_0(\beta_n) = y,
\end{displaymath}
and extended linearly. In the second equation above, $x$ denotes the multiplication-by-$x$ map $R(n) \to R(n+1)$. It follows from Proposition~\ref{prop:P-homs} that $\Phi_0$ is a well-defined equivalence of categories. Since $\cP$ and $\cQ$ contain projective generators of their respective categories, $\Phi_0$ extends to an equivalence $\Phi$ as in the statement of the theorem.
\end{proof}

\subsection{Indecomposables} \label{ss:indecomp}

A graded $R$-module is exactly a representation of the quiver
\begin{displaymath}
\xymatrix@C=3em{
\cdots \ar@<2pt>[r]^{b_{-2}} & v_{-1} \ar@<2pt>[r]^{b_{-1}} \ar@<2pt>[l]^{a_{-1}} & v_0 \ar@<2pt>[r]^{b_0} \ar@<2pt>[l]^{a_0} & v_1 \ar@<2pt>[r]^{b_1} \ar@<2pt>[l]^{a_1} & \cdots \ar@<2pt>[l]^{a_2} }
\end{displaymath}
such that the relations
\begin{displaymath}
a_{n-1} a_n=0, \qquad b_{n+1} b_n=0, \qquad a_{n+1} b_n = b_{n-1} a_n
\end{displaymath}
hold. For $n \le m$, define $I^+(n,m)$ to be the following representation. The vector space at $v_i$ is $k$ for $n \le i \le m$, and~0 elsewhere. The map $a_i$ is the identity if $n \le i \le m-1$ and $i$ has the same parity as $n$, and is~0 otherwise. The map $b_i$ is the identity if $n+1 \le i \le m$ and $i$ has the same parity as $n$, and is~0 otherwise. We define $I^-(n,m)$ similarly, but now the non-zero morphisms are those where $i$ and $n$ have opposite parity.

\begin{example}
The following illustrates $I^+(0, 4)$
\begin{center}
\begin{tikzpicture}
\node (a) at (0,2.5) {$k$};
\node (b) at (-1.5,1.5) {$k$};
\node (c) at (1.5,1.5) {$k$};
\node (d) at (3,2.5) {$k$};
\node (e) at (-3,2.5) {$k$};
\draw [->] (a) -- node[above] {$\scriptstyle a_2$} (b);
\draw [->] (a) -- node[above] {$\scriptstyle b_2$}(c);
\draw [->] (e) -- node[above] {$ \scriptstyle b_0$}(b);
\draw [->] (d) -- node[above] {$ \scriptstyle a_4$}(c);
\end{tikzpicture}
\end{center}
and the following illustrates $I^-(0, 5)$
\begin{center}
\begin{tikzpicture}
\node (a) at (0,2.5) {$k$};
\node (b) at (-1.5,1.5) {$k$};
\node (c) at (1.5,1.5) {$k$};
\node (d) at (3,2.5) {$k$};
\node (e) at (-3,2.5) {$k$};
\node (f) at (-4.5,1.5) {$k$};
\draw [->] (a) -- node[above] {$\scriptstyle a_3$} (b);
\draw [->] (a) -- node[above] {$\scriptstyle b_3$}(c);
\draw [->] (e) -- node[above] {$ \scriptstyle b_1$}(b);
\draw [->] (d) -- node[above] {$ \scriptstyle a_5$}(c);
\draw [->] (e) -- node[above] {$ \scriptstyle a_1$}(f);
\end{tikzpicture}
\end{center}
In the above diagrams, the drawn arrows are the identity, and all others vanish.
\end{example}

The modules $I^{\pm}(n,m)$ are known as \defn{zigzag modules}, and they are easily seen to be indecomposable. We note that $I^+(n,n)=I^-(n,n)$, but otherwise these modules are different. These modules appear quite often: e.g., in the theory of supergroups \cite[\S 2.2.3]{GQS} and \cite[\S 5.1]{Heidersdorf}, quivers \cite{Gabriel}, the modular representation theory of the Klein 4-group \cite[\S 2]{CM} and \cite{Johnson}, and elsewhere \cite[\S 2.1]{CdS}.

Let $J^{\pm}(n,m)$ be the object of $\cC_{\rm sp}$ corresponding to $I^{\pm}(n,m)$ under the equivalence in Theorem~\ref{thm:spequiv}. This is an indecomposable object of length $m-n+1$, with simple constituents $M(i)$ for $n \le i \le m$, each with multiplicity one.

\begin{proposition} \label{prop:indecomp}
We have the following:
\begin{enumerate}
\item Every non-projective indecomposable in $\cC_{\rm sp}$ is isomorphic to some $J^{\pm}(n,m)$.
\item The categorical dimension of $J^{\pm}(n,m)$ is $(-1)^n$ if $n \equiv m \!\pmod{2}$, and~0 otherwise.
\end{enumerate}
\end{proposition}

\begin{proof}
(a) It is shown in \cite[\S 5.2,5.3]{Germoni} that every finite dimensional non-projective indecomposable graded $R$-module is isomorphic to some $I^{\pm}(n,m)$. This implies the present statement by Theorem~\ref{thm:spequiv}. We note that \cite{Germoni} works with the ring $R'$ where $x$ and $y$ anti-commute, but graded $R$- and $R'$-modules are equivalent (change $x$ to $-x$ on odd degree elements). We also note that the indecomposable projective $R$-modules are just the rank one free modules $R(n)$.

(b) This follows from the description of the simple constituents of $J^{\pm}(m,n)$, and the fact that $M(i)$ has categorical dimension $(-1)^i$.
\end{proof}

\subsection{Decomposition of projectives} \label{s:decomposition}

From Theorems~\ref{thm:fine} and~\ref{thm:spequiv}, we know that the indecomposable projectives in $\uRep(G)$ are the generic simples and special projectives $P(n)$ for $n \in \bZ$. Recall that the Schwartz space $\cC(\bS^{\{n\}})$ (for $n \ge 1$) and the induced modules $I_{\lambda}$ are also projective. We now determine their indecomposable decompositions.

\begin{proposition} \label{prop:I-decomp}
Let $\mu$ be a weight.
\begin{enumerate}
\item In the projective object $I_{\mu}$, the multiplicity of any indecomposable projective in it (as a summand) is at most one.
\item The generic simple $M_{\lambda,\zeta}$ is a summand of $I_{\lambda}$ if and only if $\mu$ is a cyclic rotation or cyclic contraction of $\lambda$.
\item The special projective $P(n)$ (for $n \in \bZ$) is a summand of $I_{\lambda}$ if and only if $\lambda=\pi(n)$.
\end{enumerate}
\end{proposition}

\begin{proof}
The multiplicity of $M_{\lambda,\mu}$ in $I_{\mu}$ is the dimension of the space
\begin{displaymath}
\Hom_{\ul{G}}(I_{\mu}, M_{\lambda,\mu}) = \Hom_{\ul{H}}(L_{\mu}, M_{\lambda,\zeta}),
\end{displaymath}
where the isomorphism is Frobenius reciprocity. By Theorem~\ref{thm:std}, this is~1 if $\mu$ is a cyclic rotation or contraction of $\lambda$, and~0 otherwise. Similarly, the multiplicity of $P(n)$ in $I_{\mu}$ is the dimension of the space
\begin{displaymath}
\Hom_{\ul{G}}(I_{\mu}, M(n)) = \Hom_{\ul{H}}(L_{\mu}, L_{\pi(n)}),
\end{displaymath}
which is~1 if $\mu=\pi(n)$, and~0 otherwise. The result follows.
\end{proof}

Recall that $\ell(\lambda)$ is the length of $\lambda$, that $g(\lambda)$ is the order of $\Aut(\lambda)$ and that $N(\lambda) = \ell(\lambda)/g(\lambda)$. In particular, $\lambda$ consists of $N(\lambda)$ repetitions of a Lyndon word of length $g(\lambda)$.

\begin{proposition}
Let $n \ge 1$.
\begin{enumerate}
\item The multiplicity of the generic projective $M_{\lambda,\zeta}$ in $\cC(\bS^{\{n\}})$ is $g(\lambda) \binom{n}{\ell(\lambda)}$.
\item The multiplicity of the special projective $P(m)$ in $\cC(\bS^{\{n\}})$ is $\binom{n-1}{m}$.
\end{enumerate}
\end{proposition}

\begin{proof}
(a) The multiplicity of $M_{\lambda,\zeta}$ is the dimension of the space
\begin{displaymath}
\Hom_{\ul{G}}(\cC(\bS^{\{n\}}), M_{\lambda,\zeta}) = \Hom_{\ul{H}}(\cC(\bR^{(n-1)}), \ol{A}(\lambda)).
\end{displaymath}
Here we have used that $\bS^{\{n\}}$ is the induction from $H$ to $G$ of $\bR^{(n-1)}$ (see Proposition~\ref{prop:schwartz-ind}), Frobenius reciprocity, and Theorem~\ref{thm:std}. The number of terms in $\overline{A}$ coming from cyclic rotations is $g(\lambda)$ and each has multiplicity ${n-1 \choose \ell(\lambda)}$, while the number of terms coming from cyclic contractions is $g(\lambda)$ and each has multiplicity ${n-1 \choose \ell(\lambda)-1}$. Thus, 
\begin{displaymath}
\dim \Hom_{\ul{G}}(\cC(\bS^{\{n\}}), M_{\lambda,\zeta}) =  g(\lambda) {n-1 \choose \ell(\lambda)} + g(\lambda) {n-1 \choose \ell(\lambda)-1} = g(\lambda) {n \choose \ell(\lambda)}.
\end{displaymath}

(b) The multiplicity of $P(m)$ is the dimension of the space
\begin{displaymath}
\Hom_{\ul{G}}(\cC(\bS^{\{n\}}), M(m)) = \Hom_{\ul{H}}(\cC(\bR^{(n-1)}), L_{\pi(m)}),
\end{displaymath}
which is $\binom{n-1}{m}$.
\end{proof}

\section{Semisimplification} \label{s:ss}

\subsection{Background}

We now recall some generalities on quotients of tensor categories; we refer to \cite[\S 2]{EtingofOstrik} for additional details. Let $\cC$ be a $k$-linear category equipped with a symmetric monoidal structure that is $k$-bilinear. A \defn{tensor ideal} is a rule $I$ that assigns to every pair of objects $(X,Y)$ a $k$-subspace $I(X,Y)$ of $\Hom_{\cC}(X,Y)$ such that the following two conditions hold:
\begin{enumerate}
\item Let $\beta \colon X \to Y$ be a morphism in $I(X,Y)$, and let $\alpha \colon W \to X$ and $\gamma \colon Y \to Z$ be arbitrary morphisms. Then $\beta \circ \alpha$ belongs to $I(W,Y)$ and $\gamma \circ \beta$ belongs to $I(X,Z)$.
\item Let $\alpha \colon X \to Y$ be a morphism in $I(X,Y)$ and let $\alpha' \colon X' \to Y'$ be an arbitrary morphism. Then $\alpha \otimes \alpha'$ belongs to $I(X \otimes X', Y \otimes Y')$.
\end{enumerate}
Suppose that $I$ is a tensor ideal. Define a new category $\cC'$ with the same objects as $\cC$, and with
\begin{displaymath}
\Hom_{\cC'}(X,Y) = \Hom_{\cC}(X,Y)/I(X,Y).
\end{displaymath}
One readily verifies that $\cC'$ is naturally a $k$-linear symmetric monoidal category; it is called the \defn{quotient} of $\cC$ by $I$.

Suppose now that $\cC$ is pre-Tannakian. A morphism $f \colon X \to Y$ is called \defn{negligible} if $\utr(f \circ g)=0$ for any morphism $g \colon Y \to X$; here $\utr$ denotes the categorical trace. The negligible morphisms form a tensor ideal $\cN(\cC)$ of $\cC$. The quotient category is called the \defn{semisimplification} of $\cC$, and denote $\cC^{\rm ss}$. It is a semi-simple pre-Tannakian category. The simple objects of $\cC^{\rm ss}$ are the indecomposable objects of $\cC$ of non-zero categorical dimension. (Indecomposables of $\cC$ of dimension~0 become~0 in the semisimplification.)

\subsection{The main theorem}

Let $\cC=\uRep^{\rf}(G)$. The goal of \S \ref{s:ss} is to determine $\cC^{\rm ss}$. To state our result, we introduce another pre-Tannakian category $\cD$. The objects of $\cD$ are bi-graded vector spaces. For $n,m \in \bZ$, we let $k(n,m)$ be the simple object of $\cD$ concentrated in degree $(n,m)$. The tensor product on $\cD$ is the usual one, i.e.,
\begin{displaymath}
k(n,m) \otimes k(r,s) = k(n+r,m+s).
\end{displaymath}
Finally, the symmetry isomorphism is defined by
\begin{displaymath}
v \otimes w \mapsto (-1)^{\vert v \vert \cdot \vert w \vert} w \otimes v,
\end{displaymath}
where $v$ and $w$ are homogeneous and $\vert v \vert$ is the total degree of $v$. (Elements of $k(n,m)$ have total degree $n+m$.) The following is our main result:

\begin{theorem} \label{thm:semisimplification}
We have an equivalence of symmetric tensor categories $\cC^{\rm ss} \cong \cD$. Under this equivalence, $k(n,m)$ corresponds to $J^+(n-m,n+m)$ if $m \ge 0$, and to $J^-(n-m, n+m)$ if $m \le 0$.
\end{theorem}

The proof will occupy the remainder of \S \ref{s:ss}. We make one initial remark here. Since generic simples have categorical dimension~0, they become~0 in $\cC^{\rm ss}$. We have seen (Proposition~\ref{prop:indecomp}) that the indecomposables of $\cC_{\rm sp}$ are the modules $J^{\pm}(n,m)$, and that this module has categorical dimension~0 if and only if $n$ and $m$ have opposite parity. Thus the simples of $\cC^{\rm ss}$ are exactly the objects $J^{\pm}(n,m)$ with $n \equiv m \!\pmod{2}$. In particular, the additive functor $\cD \to \cC^{\rm ss}$ defined by the correspondence in the theorem is an equivalence of abelian categories. The main content of the theorem concerns the tensor structures.

\begin{remark}
The category $\cD$ is super-Tannakian. In the notation of \cite[\S 0.3]{Deligne2}, $\cD=\Rep(G, \epsilon)$ where $G=\bG_m \times \bG_m$ and $\epsilon=(-1,-1)$.
\end{remark}

\begin{remark}
The semi-simplification of $\cC$ is very similar to that of the supergroup $\GL(1|1)$; see \cite[Theorem~5.12]{Heidersdorf}.
\end{remark}

\subsection{The key computation} \label{ss:sskey}

Much of the proof of Theorem~\ref{thm:semisimplification} follows from general considerations and results we have already established. However, we will require the following genuinely new piece of information:

\begin{proposition} \label{prop:M-tensor}
For $n,m \in \bZ$, we have $M(n) \otimes M(m) \cong M(n+m)$ in $\cC^{\rm ss}$.
\end{proposition}

We require a few lemmas.

\begin{lemma} \label{lem:M-tensor-1}
Let $b \in \bS^{\{n+1\}}$, and let $\tau$ be the generator of $G[b]/G(b)$. Then $M(1)^{G(b)}$ is $n$-dimensional, and the eigenvalues of $\tau$ on this space are the $(n+1)$st roots of unity, except for~1.
\end{lemma}

\begin{proof}
We have $\Delta(1) = \bbone \oplus L_{\wa}$ as an $\ul{H}$-module, and so $\Delta(1)^{G(b)}$ is $(n+1)$-dimensional by \cite[Corollary~4.11]{line}. For $a \in \bS$, let $\ev_a \colon \cC(\bS) \to k$ be the evaluation at $a$ map, as in Lemma~\ref{lem:std-7}. One easily sees that these functionals are linearly independent (see \cite[Proposition~4.6]{line}). Thus the functionals $\ev_{b_i}$ for $1 \le i \le n+1$ form a basis for the $\ul{G}(b)$-coinvariants of $\Delta(1)$, which is isomorphic to $\Delta(1)^{G(b)}$. The action of $\tau$ permutes the basis vectors of this space, and so $\Delta(1)^{G(b)}$ is the regular representation of $G[b]/G(b) \cong \bZ/(n+1)$. Since $M(1)=\Delta(1)/\bbone$, we have $M(1)^{G(b)}=\Delta(1)^{G(b)}/\bbone^{G(b)}$. Since $\bbone^{G(b)}$ is 1-dimensional with trivial $\tau$ action, the result follows.
\end{proof}

\begin{lemma} \label{lem:M-tensor-2}
For $n>0$ we have an isomorphism
\begin{displaymath}
M(1) \otimes M(n) = M(n+1) \oplus \bigoplus_{\zeta \in \mu_{n+1} \setminus \{ \epsilon(n+1) \}} M_{\pi(n+1), \zeta}
\end{displaymath}
in $\cC$, where $\mu_{n+1}$ denotes the set of $(n+1)$st roots of unity. In particular, we have an isomorphism $M(1) \otimes M(n) \cong M(n+1)$ in $\cC^{\rm ss}$. 
\end{lemma}

\begin{proof}
We have $M(1)=L_{\wb}$ and $M(n)=L_{\pi(n)}$ as $H$-representations, and so
\begin{displaymath}
M(1) \otimes M(n) = L_{\pi(n+1)}^{\oplus n+1} \oplus L_{\pi(n)}^{\oplus n}
\end{displaymath}
by \cite[Theorem~7.2]{line} (or, better, \cite[Proposition~7.5]{line}). In particular, for $b \in \bS^{\{n+1\}}$, we have
\begin{displaymath}
\dim M(1)^{G(b)} = n, \qquad \dim M(n)^{G(b)} = 1, \qquad \dim (M(1) \otimes M(n))^{G(b)} = n.
\end{displaymath}
It follows that the natural map
\begin{displaymath}
M(1)^{G(b)} \otimes M(n)^{G(b)} \to (M(1) \otimes M(n))^{G(b)}
\end{displaymath}
is an isomorphism. Since $\tau$ acts by $-\epsilon(n)$ on $M(n)^{G(b)}$ by Theorem~\ref{thm:fine}, Lemma~\ref{lem:M-tensor-1} shows that the $\tau$-eigenvalues on the right side above are exactly the $(n+1)$st roots of unity except for $-\epsilon(n)=\epsilon(n+1)$. In particular, this shows that $M(n)$ is not a constituent of $M(1) \otimes M(n)$. Thus each $L_{\pi(n)}$ in the tensor product belongs to a generic simple $M_{\pi(n+1),\zeta}$. Looking at the $\tau$ eigenvalues on this space (from Theorem~\ref{thm:std}), we see that each of these simples occurs exactly once. (Note that the generic condition excludes $\zeta=\epsilon(n+1)$.) Since generic simples are projective, these all split off from the tensor product as summands, that is, we have
\begin{displaymath}
M(1) \otimes M(n) = X \oplus \bigoplus_{\zeta \in \mu_{n+1} \setminus \{\epsilon(n+1)\}} M_{\pi(n+1), \zeta}
\end{displaymath}
as $\ul{G}$-modules. Since the $\ul{H}$-module underlying $X$ is $L_{\pi(n+1)}$, we must have $X=M(n+1)$. Since the generic simples become~0 in $\cC^{\rm ss}$, the result follows.
\end{proof}

\begin{lemma} \label{lem:M-tensor-3}
We have an isomorphism
\begin{displaymath}
M(1) \otimes M(-1) \cong \bbone \oplus M_{\wa\wb}
\end{displaymath}
in $\cC$, and thus an isomorphism $M(1) \otimes M(-1) \cong M(0)$ in $\cC^{\rm ss}$.
\end{lemma}

\begin{proof}
We have $M(1)=L_{\wa}$ and $M(-1)=L_{\wb}$ as $\ul{H}$-modules. Thus the $\ul{H}$-module underlying $M(1) \otimes M(-1)$ is
\begin{displaymath}
L_{\wa\wb} \oplus L_{\wb\wa} \oplus L_{\wa} \oplus L_{\wb} \oplus \bbone
\end{displaymath}
by \cite[Theorem~7.2]{line} (or, better, \cite[Lemma~7.7]{line}). From the classification of simples in $\uRep(G)$, we see that $M_{\wa\wb}$ must be a constituent of the tensor product. Since this is a generic simple, it splits off, and so $M(1) \otimes M(-1) \cong X \oplus M_{\wa\wb}$ for some $\ul{G}$-module $X$. The $\ul{H}$-module underlying $X$ is trivial, and so $X$ is trivial as a $\ul{G}$-module. Since $M_{\wa\wb}$ becomes~0 in $\cC^{\rm ss}$, the result follows.
\end{proof}

\begin{proof}[Proof of Proposition~\ref{prop:M-tensor}]
We must show
\begin{displaymath}
M(n) \otimes M(m) \cong M(n+m)
\end{displaymath}
holds in $\cC^{\rm ss}$ for all $n,m \in \bZ$. If $n$ and $m$ are both non-negative, this follows from Lemma~\ref{lem:M-tensor-2} by induction. Dualizing gives the case when $n$ and $m$ are both non-positive. Now suppose $n>0$ and $m<0$. We have
\begin{displaymath}
M(n) \otimes M(m) \cong (M(n-1) \otimes M(1)) \otimes (M(-1) \otimes M(m+1)) \cong M(n-1) \otimes M(m+1),
\end{displaymath}
where we first used the cases already established, and then used Lemma~\ref{lem:M-tensor-3}. Continuing by induction, we obtain the desired result.
\end{proof}

\subsection{Heller shifts} \label{ss:heller}

Let $X$ and $Y$ be objects of $\cC$. Define $X \sim Y$ if there exist projective objects $P$ and $Q$ such that $X \oplus P \cong Y \oplus Q$. This is an equivalence relation on objects of $\cC$. Now, choose a surjection $\phi \colon P \to X$, with $P$ projective, and an injection $\psi \colon X \to I$, with $I$ injective. (Note: injectives and projectives are the same in $\cC$.) We define the \defn{Heller shifts} of $M$ by $\Omega(X)=\ker(\phi)$ and $\Omega^{-1}(X)=\coker(\psi)$. These are well-defined up to equivalence by Schanuel's lemma; see \cite[Lemma~1.5.3]{Benson} and the following discussion. Note that all projectives in $\cC$ have categorical dimension~0, so $X \sim Y$ implies $X \cong Y$ in $\cC^{\rm ss}$.

\begin{proposition} \label{prop:heller}
Let $X$ and $Y$ be objects of $\cC$, and let $n,m \in \bZ$. Then:
\begin{enumerate}
\item If $X \sim Y$ then $\Omega(X) \sim \Omega(Y)$.
\item $\Omega^{-1}(\Omega(X)) \sim X$ and $\Omega(\Omega^{-1}(X)) \sim X$.
\item $X \otimes \Omega(Y) \sim \Omega(X \otimes Y)$, and similarly for $\Omega^{-1}$.
\item $\Omega^n(X) \otimes \Omega^m(Y) \sim \Omega^{n+m}(X \otimes Y)$.
\end{enumerate}
\end{proposition}

\begin{proof}
(a) is clear.

(b) Consider a short exact sequence
\begin{displaymath}
0 \to \Omega(X) \to P \to X \to 0
\end{displaymath}
with $P$ projective. Since $P$ is also injective this shows that $X \sim \Omega^{-1}(\Omega(X))$. The other direction is similar.

(c) Consider a short exact sequence as above. Tensoring with $Y$, we find
\begin{displaymath}
0 \to \Omega(X) \otimes Y \to P \otimes Y \to X \otimes Y \to 0.
\end{displaymath}
Since $P \otimes Y$ is projective \cite[Proposition~4.2.12]{EGNO}, it follows that the kernel above is $\Omega(X \otimes Y)$.

(d) We have
\begin{displaymath}
\Omega^n(X) \otimes \Omega^m(Y) \sim \Omega^m(\Omega^n(X) \otimes Y) \sim \Omega^m(\Omega^n(X \otimes Y)) \sim \Omega^{n+m}(X \otimes Y)
\end{displaymath}
where in the first two steps we used (c), and in the final step we used (b) (if $n$ and $m$ have different signs).
\end{proof}

\begin{proposition} \label{prop:heller-J}
For $n \in \bZ$ and $m \ge 0$, we have, in $\cC^{\rm ss}$,
\begin{displaymath}
\Omega^m(M(n)) \cong J^+(n-m, n+m), \qquad \Omega^{-m}(M(n)) \cong J^-(n-m, n+m)
\end{displaymath}
In particular, every simple object of $\cC^{\rm ss}$ is isomorphic to $\Omega^m(M(n))$ for a unique $(n,m) \in \bZ^2$.
\end{proposition}

\begin{proof}
By Theorem~\ref{thm:spequiv}, it is equivalent to prove the analogous result for $R$-modules. This is well-known: for example, the same argument used to prove \cite[Theorem~1]{Johnson} applies here.
\end{proof}

\begin{proposition} \label{prop:heller-tensor}
For integers $n$, $m$, $r$, and $s$, we have an isomorphism in $\cC^{\rm ss}$
\begin{displaymath}
\Omega^r(M(m)) \otimes \Omega^s(M(n)) \cong \Omega^{r+s}(M(m+n)).
\end{displaymath}
\end{proposition}

\begin{proof}
This follows from Propositions~\ref{prop:M-tensor} and~\ref{prop:heller}. (Note that the proof of Proposition~\ref{prop:M-tensor} actually shows $M(n) \otimes M(m) \sim M(n+m)$.)
\end{proof}

\subsection{Pointed tensor categories} \label{ss:pointed}

The simple objects of $\cC^{\rm ss}$ and $\cD$ can both be indexed by $\bZ^2$, and Proposition~\ref{prop:heller-tensor} shows that tensor products of simples decompose in the same way in both categories. To actually construct an equivalence of symmetric tensor categories directly would require a good deal of additional work. We circumvent this by appealing to the general theory of pointed tensor categories.

Let $\cT$ be a semisimple $k$-linear pre-Tannakian category. We say that $\cT$ is \defn{pointed} if its simple objects are invertible, i.e., given a simple object $S$ we have $S \otimes S' \cong \bbone$ for some $S'$ (necessarily the dual of $S$). Suppose this is the case. It follows that the tensor product of two simples is again simple. Let $A(\cT)$ be the set of isomorphism classes of simple objects. This forms an abelian group under tensor product.

The group $A(\cT)$ carries a natural piece of extra structure. Let $S$ be a simple object. Then $S \otimes S$ is a simple object, and so (by Schur's lemma) the symmetry isomorphism $\beta_{S,S}$ must be multiplication by a scalar. Since our category is symmetric (as opposed to braided), this scalar must be $\pm 1$. We can thus define a function
\begin{displaymath}
q_{\cT} \colon A(\cT) \to \bZ/2, \qquad (-1)^{q_{\cT}(S)} = \beta_{S,S}.
\end{displaymath}
One easily sees that $q_{\cT}$ is in fact a group homomorphism. (In the general braided theory, $q$ will be a quadratic form valued in $k^{\times}$, but in the symmetric case the situation simplifies; see \cite[\S 2.11]{DGNO} and \cite[Example~2.45]{DGNO}.) The following is the main result we require:

\begin{theorem} \label{thm:pointed}
Let $\cT$ and $\cT'$ be semi-simple pointed pre-Tannakian categories. Given an isomorphism
\begin{displaymath}
\phi \colon (A(\cT), q_{\cT}) \to (A(\cT'), q_{\cT'})
\end{displaymath}
there exists an equivalence $\Phi \colon \cT \to \cT'$ of symmetric tensor categories inducing $\phi$; moreover, $\Phi$ is unique up to isomorphism.
\end{theorem}

\begin{proof}
When $A(\cT)$ is finite this is the symmetric case of \cite[Proposition~2.41]{DGNO} as explained by \cite[Example~2.45]{DGNO}. But the assumption that $A(\cT)$ be finite is unnecessary as is clear from the more general version of \cite[Proposition~2.41]{DGNO} given in \cite[Theorem 3.3]{Joyal-Street}.
\end{proof}

There is one other general fact that will be useful:

\begin{proposition} \label{prop:pointed-q}
Let $\cT$ be a semi-simple pointed pre-Tannakian category, and let $S$ be a simple object. Then the categorical dimension of $S$ is $q_{\cT}([S])$.
\end{proposition}

\begin{proof}
Put $A=A(\cT)$ and $q=q_{\cT}$. Let $\cT'$ be the category of $A$-graded vector spaces, with its usual tensor product. Define a symmetric structure on $\cT'$ by
\begin{displaymath}
v \otimes w \mapsto (-1)^{q(\vert v \vert) q(\vert w \vert)} w \otimes v,
\end{displaymath}
where $v$ and $w$ are homogeneous elements, and $\vert v \vert \in A$ denotes the degree of $v$. Then one readily verifies that $A(\cT')=A$ and $q_{\cT'}=q$, and so by Theorem~\ref{thm:pointed} there is an equivalence $\cT \cong \cT'$. Letting $k(a)$ denote the simple of $\cT'$ concentrated in degree $a$, one easily sees that the categorical dimension of $k(a)$ is $q(a)$, which completes the proof.
\end{proof}

\begin{proof}[Proof of Theorem~\ref{thm:semisimplification}]
Theorem~\ref{prop:heller-tensor} shows that $\cC^{\rm ss}$ is a pointed tensor category. Indeed, every simple of $\cC^{\rm ss}$ is of the form $\Omega^m(M(n))$ for $n,m \in \bZ$, and the proposition shows
\begin{displaymath}
\Omega^m(M(n)) \otimes \Omega^{-m}(M(-n)) \cong \bbone.
\end{displaymath}
The same proposition shows that the map $\bZ^2 \to A(\cC^{\rm ss})$ given by $(n,m) \mapsto [\Omega^m(M(n))]$ is a group isomorphism. By Proposition~\ref{prop:indecomp} and~\ref{prop:heller-J}, the object $\Omega^m(M(n))$ has categorical dimension $(-1)^{n+m}$, and so $q_{\cC^{\rm ss}}(n,m) = n+m$ by Proposition~\ref{prop:pointed-q}.

We also have an isomorphism $\bZ^2 \to A(\cD)$ by $(n,m) \mapsto [k(n,m)]$. One readily verifies that $q_{\cD}(n,m)=n+m$. Thus by Theorem~\ref{thm:pointed}, we have an equivalence of symmetric tensor categories $\cD \to \cC^{\rm ss}$ which maps $k(n,m)$ to $\Omega^m(M(n))$.
\end{proof}

\section{Delannoy loops} \label{s:loops}

In \cite[\S 3.4, \S 9]{line} we gave a description of the category $\uRep(H)$ in terms of Delannoy paths. In this section we give a similar description of $\uRep(G)$ in terms of a natural generalization that we call Delannoy loops.

\subsection{Delannoy loops}

An \defn{$(n,m)$-Delannoy path} is a path in the plane from $(0,0)$ to $(n,m)$ composed of steps of the form $(1,0)$, $(0,1)$, and $(1,1)$. The \defn{Delannoy number} $D(n,m)$ is the number of $(n,m)$-Delannoy paths, and the \defn{central Delannoy number} $D(n)$ is $D(n,n)$. For example, $D(2)=13$; see Figure~\ref{fig:delannoy}. The Delannoy numbers are well-known in the literature; see, e.g., \cite{Banderier}.

\begin{figure}
\def\delannoy#1{\begin{tikzpicture}[scale=0.5]
\draw[step=1, color=gray!50] (0, 0) grid (2,2);
\draw[line width=2pt] #1;
\end{tikzpicture}}
\begin{center}
\delannoy{(0,0)--(1,0)--(2,0)--(2,1)--(2,2)} \quad
\delannoy{(0,0)--(1,0)--(2,1)--(2,2)} \quad
\delannoy{(0,0)--(1,0)--(1,1)--(2,1)--(2,2)} \quad
\delannoy{(0,0)--(1,1)--(2,1)--(2,2)} \quad
\delannoy{(0,0)--(1,0)--(1,1)--(2,2)} \quad
\delannoy{(0,0)--(1,0)--(1,1)--(1,2)--(2,2)} \quad
\delannoy{(0,0)--(1,1)--(2,2)}
\end{center}
\vskip.5\baselineskip
\begin{center}
\delannoy{(0,0)--(0,1)--(1,1)--(2,1)--(2,2)} \quad
\delannoy{(0,0)--(0,1)--(1,1)--(2,2)} \quad
\delannoy{(0,0)--(1,1)--(1,2)--(2,2)} \quad
\delannoy{(0,0)--(0,1)--(1,1)--(1,2)--(2,2)} \quad
\delannoy{(0,0)--(0,1)--(1,2)--(2,2)} \quad
\delannoy{(0,0)--(0,1)--(0,2)--(1,2)--(2,2)}
\end{center}
\caption{The thirteen $(2,2)$-Delannoy paths.}
\label{fig:delannoy}
\end{figure}
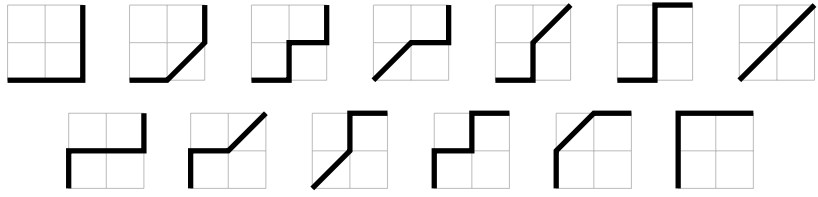

An \defn{$(n,m)$ Delannoy loop} is an oriented unbased loop on a toroidal grid with points labeled by $\bZ/n \times  \bZ/m$ composed of steps of the form $(1,0)$, $(0,1)$, and $(1,1)$, and which loops around the torus exactly once in each of the $x$-direction and the $y$-direction. Loops are allowed to touch themselves at vertices (though it turns out that this can happen at most once), and are not required to pass through $(0,0)$.

The set of Delannoy loops will be denoted $\Lambda(n,m)$. The \defn{circular Delannoy number} $C(n,m)$ is the number of $(n,m)$-Delannoy loops, and the \defn{central circular Delannoy number} $C(n)$ is $C(n,n)$. For example, $C(2)=16$; see Figure~\ref{fig:circular-delannoy}. Although we will see that circular Delannoy numbers have a number of natural combinatorial properties, they have not previously appeared in the OEIS and do not seem to have been studied elsewhere.

The group $\bZ/n \times \bZ/m$ acts on Delannoy loops by cyclically permuting the coordinates. The actions of $\bZ/n$ and $\bZ/m$ individually are free, but the product action is not free.

\begin{figure}
\def\delannoy#1{\begin{tikzpicture}[scale=0.5]
\draw[step=1, color=gray!50] (0, 0) grid (2,2);
\draw[line width=2pt] #1;
\end{tikzpicture}}
\begin{center}
\delannoy{(0,0)--(1,0)--(2,0)--(2,1)--(2,2)} \quad
\delannoy{(0,0)--(1,0)--(1,1)--(1,2)--(2,2)}, \quad
\delannoy{(0,0)--(0,1)--(1,1)--(2,1)--(2,2)} \quad
\delannoy{(1,0)--(1,1)--(2,1) (0,1)--(1,1)--(1,2)}, \quad
\delannoy{(0,0)--(1,0)--(2,1)--(2,2)} \quad
\delannoy{(0,0)--(1,1)--(1,2)--(2,2)}, \quad
\delannoy{(0,0)--(0,1)--(1,1)--(2,2)} \quad
\delannoy{(0,1)--(1,2) (1,0)--(1,1)--(2,1)}, \quad
\end{center}
\vskip.5\baselineskip
\begin{center}
\delannoy{(0,0)--(1,1)--(2,1)--(2,2)} \quad
\delannoy{(0,1)--(1,1)--(1,2)  (1,0)--(2,1)}, \quad
\delannoy{(0,0)--(1,0)--(1,1)--(2,2)} \quad
\delannoy{(0,0)--(0,1)--(1,2)--(2,2)}, \quad
\delannoy{(0,0)--(1,1)--(2,2)} \quad
\delannoy{(0,1)--(1,2) (1,0)--(2,1)}, \quad
\delannoy{(0,0)--(1,0)--(1,1)--(2,1)--(2,2)} \quad
\delannoy{(0,0)--(0,1)--(1,1)--(1,2)--(2,2)} \quad
\end{center}
\caption{The 16 $(2,2)$-Delannoy loops, paired up by the $\bZ/2$-action on the $x$-coordinate. These form five orbits under $\bZ/2 \times \bZ/2$, with the first and second pairs, the third and fourth pairs, and and the fifth and sixth pairs each combining into orbits of size $4$, and the seventh and eigth pairs each forming their own orbit of size $2$.}
\label{fig:circular-delannoy}
\end{figure}
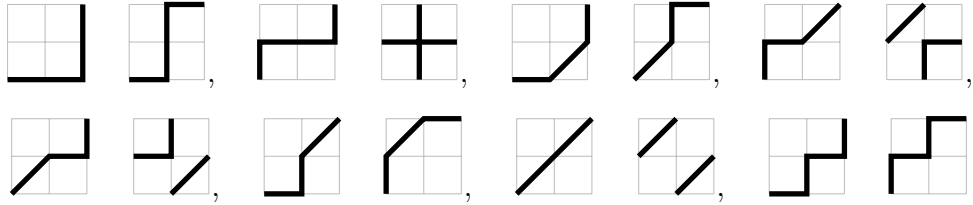

\subsection{Orbits}

As in \cite[Proposition~3.5]{line}, the $H$-orbits on $\bR^{(n)} \times \bR^{(m)}$ are naturally in bijective correspondence with $(n,m)$-Delannoy paths. In particular, the number of orbits is the Delannoy number $D(n,m)$. Here is a description of a bijection. Given a point $(x,y) \in \bR^{(n)} \times \bR^{(m)}$ we break up $\bR$ into the points $x \cup y$ and intervals $\bR - (x \cup y)$. For each interval $I_k$ we have a maximal $a_k$ with $x_{a_k} < I_k$ and a maximal $b_k$ with $x_{b_k} <I$, except for the leftmost interval where we define $a_0 = b_0 = 0$. Then look at the path which goes from $(0,0) = (a_0,b_0)$ to $(a_1, b_1)$, etc. Between each interval there is a point that lies in $x \cup y$ and hence lies in $x$, in $y$, or in both. From the description of the path it is clear that it steps by $(1,0)$ when we go past a point in just $x$, by $(0,1)$ when we go past a point in just $y$, and by $(1,1)$ when we go past a point in both. Hence this path is a Delannoy path. See \cite[Figure 2]{line} for an example.

We now prove the analogous result in the circle case.

\begin{proposition}
The $G$-orbits on $\bS^{\{n\}} \times \bS^{\{m\}}$ are in bijective correspondence with $(n,m)$-Delannoy loops. This correspondence is natural, and, in particular, equivariant for $\bZ/n \times \bZ/m$. Hence, the number of orbits is the circular Delannoy number $C(n,m)$.
\end{proposition}

\begin{proof}
Suppose that $(x,y)$ is an element of $\bS^{\{n\}} \times \bS^{\{m\}}$. Again $\bS - (x \cup y)$ is a collection of intervals. To each interval we again assign the point $(a, b)$ where $x_a$ is the largest of the $x$ which is smaller than $I$ in cyclic ordering and similarly for $y_b$. Again between intervals we pass through a point which is either in $x$, in $y$, or in both, and we step by $(1,0)$, $(0,1)$, or $(1,1)$ respectively. Since each point in $x$ and each point in $y$ is crossed exactly once, this loop will go around the torus exactly once in each of the $x$ and $y$ directions, so this gives a Delannoy loop $p(x,y)$. The proof that this gives a bijection is the same as in the line case.
\end{proof}

\begin{remark}
If we fix a vertex $(a,b)$, then the intervals assigned to $(a,b)$ are exactly the intersection of the intervals $(x_a,x_{a+1})$ and $(y_b, y_{b+1})$. Such an intersection can be empty, a single interval, or a union of two intervals. This is why it's possible for a Delannoy loop to intersect itself at a vertex. In the $(2,2)$ case, this phenomenon occurs for a single $\bZ/2 \times \bZ/2$ orbit, consisting of the first two pairs in Figure \ref{fig:circular-delannoy}.
\end{remark}

\subsection{Representations}

As in the line case, we can describe an equivalent category to $\uPerm^{\circ}(G)$ directly using Delannoy loops. Since the proof is essentially the same as in the line case we content ourselves with giving a correct statement.
 
Let $\cC(n,m)$ be the vector space with basis indexed by the set of Delannoy loops $\Lambda(n,m)$. We write $[p]$ for the basis vector of $\cC(n,m)$ corresponding to $p \in \Lambda(n,m)$. We will define a composition law on the $\cC$'s. Let $p_1 \in \cC(n,m)$, $p_2 \in \cC(m,\ell)$, and $p_3 \in \cC(n,\ell)$ be given. 

We define $(n,m,\ell)$ Delannoy loops similarly to above: they are loops in a toroidal grid labeled by $\bZ/n \times \bZ/m \times \bZ/\ell$ that can take steps which increase by one in any non-empty subset of the three coordinates.  If $q$ is in $\Lambda(n,m,\ell)$, we can use the projection maps $\pi_{i,j}$ to produce $\pi_{1,2}(q) \in \Lambda(n,m)$, $\pi_{1,3}(q) \in \Lambda(n,\ell)$, and $\pi_{1,2}(q) \in \Lambda(m,\ell)$ (see \cite[\S 9.1]{line} for details). We say that $q$ is \emph{compatible with} a triple $p_{1,2}, p_{2,3}, p_{1,3}$ of Delannoy loops if 
\begin{displaymath}
\pi_{1,2}(q)=p_{1,2}, \quad \pi_{2,3}(q)=p_{2,3}, \quad \pi_{1,3}(q)=p_{1,3},
\end{displaymath}

Let
\begin{displaymath}
\epsilon(q, p_1,p_2,p_3) = \begin{cases}
(-1)^{\ell(q)+\ell(p_3)} & \text{if $q$ is compatible with $(p_1,p_2, p_3)$} \\
0 & \text{otherwise} \end{cases}
\end{displaymath}

We define a $k$-bilinear composition law
\begin{displaymath}
\cC(n,m) \times \cC(m,\ell) \to \cC(n,\ell)
\end{displaymath}
by
\begin{displaymath}
[p_1] \circ [p_2] = \sum_{q \in \Gamma(n,m,\ell), p_3 \in \Gamma(n,\ell)} \epsilon(q,p_1,p_2,p_3) [p_3].
\end{displaymath}

\begin{proposition} \label{prop:equiv}
We have the following:
\begin{enumerate}
\item The composition law in $\cC$ is associative and has identity elements.
\item There is a fully faithful functor $\Phi \colon \cC \to \uRep(G)$ defined on objects by $\Phi(X_n)=\cC(\bS^{\{n\}})$ and on morphisms by $\Phi([p])=A_p$.
\item $\Phi$ identifies $\cC$ with the full subcategory of $\uRep(G)$ spanned by the $\cC(\bS^{\{n\}})$'s.
\item $\Phi$ identifies the additive envelope of $\cC$ with $\uPerm^{\circ}(G)$.
\end{enumerate}
\end{proposition}

\begin{remark}
In the line case it turned out that given Delannoy paths $p_1$ and $p_2$, and $p_3$ there is at most one $q$ such that $q$ is compatible with $(p_1,p_2,p_3)$, and hence each $[p_3]$ only appeared once in the sum. In the circle setting it is possible for $(p_1,p_2,p_3)$ to be compatible with \emph{two} such $q$. For example, there exactly two $(1,1,1)$ Delannoy loops that have no diagonal steps (one for each of the cyclic orderings on the coordinates, and they both project down in all three directions to the unique non-diagonal $(1,1)$ loop.
\end{remark}

\begin{remark}
In the line case, we also had that $\Phi$ identifies the additive--Karoubian envelope of the path category with $\uRep^{\rf}(H)$. In the circle case $\uRep^{\rf}(G)$ is not semisimple, and so one needs to be more careful about what kind of completion to apply.
\end{remark}

\subsection{Enumeration}

We have the following combinatorial formula for circular Delannoy numbers in terms of ordinary Delannoy numbers. We give two proofs of this formula, one by counting and one using representation theory. 

\begin{proposition}
If $n,m>0$ then we have 
\begin{align*}
C(n,m) &= n\left(D(n,m-1) + D(n-1,m-1)\right) \\
&= n\left(D(n,m) - D(n-1,m)\right)
\end{align*}
\end{proposition}

The two formulas are equal to each other using the standard recurrence relation for ordinary Delannoy numbers, $D(n,m) = D(n,m-1) + D(n-1,m) + D(n-1,m-1)$. We will prove the first version.

\begin{proof}[Counting proof]
Each orbit of $\bZ/n$ contains a unique representative which passes through $(0,0)$ and immediately takes a vertical or diagonal step. If we ignore this first step then we obtain an ordinary $(n,m-1)$-Delannoy path if the first step is vertical, and an ordinary $(n-1,m-1)$-Delannoy path if the first step is diagonal.
\end{proof}

\begin{proof}[Representation theoretic proof]
Second we give the representation theoretic proof, using that Delannoy loops form a basis for $\Hom_{\ul{G}}(\cC(\bS^{\{n\}}), \cC(\bS^{\{m\}}))$. Recall that $\bS^{\{k\}} \cong \Ind_H^G \cC(\bR^{(k-1)})$ (see the proof of Proposition~\ref{prop:schwartz-ind}) and
\begin{displaymath}
\Res_H^G (\cC(\bS^{\{k\}})) \cong \cC(\bR^{(k)})^{\oplus k} \oplus  \cC(\bR^{(k-1)})^{\oplus k}
\end{displaymath}
(see \S \ref{ss:std-setup}). Now we use Frobenius reciprocity, and the fact that $\Hom( \cC(\bR^{(a)}),  \cC(\bR^{(b)}))$ has dimension $D(a,b)$ to compute
\begin{align*}
C(n,m) &= \dim \Hom_{\ul{G}}(\cC(\bS^{\{n\}}), \cC(\bS^{\{m\}})) \\
&= \dim \Hom_{\ul{G}}(\cC(\bS^{\{n\}}), \Ind_H^G \cC(\bR^{(m-1)})) \\ 
&= \dim \Hom_{\ul{H}}(\Res_H^G \cC(\bS^{\{n\}}), \cC(\bR^{(m-1)})) \\
&= \dim \Hom_{\ul{H}}(\cC(\bR^{(n)})^{\oplus n} \oplus \cC(\bR^{(n-1)})^{\oplus n}, \cC(\bR^{(m-1)})) \\
&= n\left(D(n,m-1) + D(n-1,m-1)\right) \qedhere
\end{align*}
\end{proof}

\begin{figure}
$$\begin{array}{c|cccccccccc}
m \backslash n& 0 & 1 & 2 & 3 & 4 & 5 & 6 & 7 & 8 & 9 \\ \hline
0 & 1 &  1 & 1 &  1 & 1 & 1 &  1 & 1 & 1 & 1\\
1 & 1 &  2& 4& 6& 8& 10& 12& 14& 16& 18\\
2 & 1 &  4& 16& 36& 64& 100& 144& 196& 256& 324\\
3 & 1 &  6& 36& 114& 264& 510& 876& 1386& 2064& 2934\\
4 & 1 &  8& 64& 264& 768& 1800& 3648& 6664& 11264& 17928\\
5 & 1 & 10& 100& 510& 1800& 5010& 11820& 24710& 47120& 83610\\
6 & 1 & 12& 144& 876& 3648& 11820& 32016& 75852& 162048& 318924 \\
7 & 1 & 14& 196& 1386& 6664& 24710& 75852& 201698& 479248& 1040382\\
8 & 1 & 16& 256& 2064& 11264& 47120& 162048& 479248& 1257472& 2994192\\
9 & 1 & 18& 324& 2934& 17928& 83610& 318924& 1040382& 2994192& 7777314\\
\end{array}$$
\caption{A table of values of the circular Delannoy numbers.}
\label{fig:circular-table}
\end{figure}

Using the proposition, we can easily compute $C(n,m)$ for small $n$ and $m$, as shown in Figure \ref{fig:circular-table}. We can also extract a nice formula for $C(n,m)$, using the representation-theoretic formula for the ordinary Delannoy numbers (see \cite[Remark~4.9]{line}):
\begin{displaymath}
D(n,m) = \sum_{k=0}^{\min(m,n)} \binom{n}{k} \binom{m}{k} 2^k.
\end{displaymath}

\begin{proposition} \label{prop:circular-sum-formula}
If $n,m \geq 1$ then
\begin{displaymath}
C(n,m) = \sum_{k=1}^{\min(m,n)} \binom{n}{k} \binom{m}{k} k 2^k
\end{displaymath}
\end{proposition}

\begin{proof}
Letting $N=\min(n,m)$, we have
\begin{align*}
C(n,m) = & n\left(D(n,m) - D(n-1,m)\right) \\
&= n \sum_{k=1}^N \bigg[ \binom{n}{k} \binom{m}{k} 2^k - \binom{n-1}{k} \binom{m}{k} 2^k \bigg] \\
&= n \sum_{k=1}^N \binom{n-1}{k-1} \binom{m}{k}  2^k = \sum_{k=1}^N \binom{n}{k} \binom{m}{k} k 2^k
\end{align*}
Note that the $k=0$ terms in the sums defining the two Delannoy numbers cancel, which is why we can start from $k=1$.
\end{proof}

This formula for $C(n,m)$ can also be proved directly using the decomposition of $\cC(\bS^{\{n\}})$ into indecomposable projectives from \S \ref{s:decomposition} and counting the dimensions of the Hom spaces between projectives from \S \ref{s:maps-of-projectives}. This derivation is somewhat messy due to the contributions from the special block, so we do not work it out here.

The recurrence for ordinary Delannoy numbers immediately yields the generating function $(1-x-y-xy)^{-1}$, and from our formula for $C(n,m)$ in terms of $D(n,m)$ it is easy to derive the following generating function for circular Delannoy numbers with $n,m \geq 1$.

\begin{displaymath}
\sum_{n,m=1}^\infty C(n,m) x^n y^m = \frac{2xy}{(1-x-y-xy)^2}
\end{displaymath}

\end{document}